\documentclass[11pt]{amsart}
\usepackage[vmargin=1.2in,hmargin=1.2in]{geometry}
\title[A Finitary Approach to Coarse Separation of Euclidean spaces]{A finitary Approach to Coarse Separation of Euclidean spaces}

\author{Harsh Patil}
\address{School of Mathematics, University of Bristol, Bristol, UK.}
\email{cr22307@bristol.ac.uk}
\date{\today}

\usepackage{amsmath,amsthm,amsfonts,graphicx,amssymb}
\usepackage{hyperref}
\usepackage{xcolor}
\usepackage{epigraph}
\usepackage{tikz}
\usepackage{tikz-cd}
\usepackage{comment}
\usepackage{float}
\usepackage{comment}

\numberwithin{equation}{section}
\newtheorem{theorem}[equation]{Theorem}

\newtheorem{corollary}[equation]{Corollary}
\newtheorem{lemma}[equation]{Lemma}

\theoremstyle{definition}

\newtheorem{definition}[equation]{Definition}
\newtheorem{remark}[equation]{Remark}

\newtheorem*{theorem*}{Theorem}

\newtheoremstyle{citing}
  {3pt}
  {3pt}
  {\itshape}
  {}
  {\bfseries}
  {}
  {.5em}
  {\thmnote{#3}}

\theoremstyle{citing}

\newcommand{\ori}{\partial_{0}}
\newcommand{\ter}{\partial_{1}}

\def\XXint#1#2#3{{\setbox0=\hbox{$#1{#2#3}{\int}$}
\vcenter{\hbox{$#2#3$}}\kern-.5\wd0}}

\numberwithin{equation}{section}
\begin{document}
\maketitle

\begin{center}
    \textit{We are usually convinced more
easily by reasons we have found
ourselves than by those which have
occurred to others... \vspace{5mm} -Pascal}

\end{center}

\begin{abstract}
   We give a novel proof of the fact that every coarsely separating family of subsets of the Euclidean space $\mathbb{R}^{d}$ must have asymptotic dimension at least $d-1$. The proof only uses singular homology/cohomology and standard facts from algebraic topology such as Alexander duality.  We do this by first reducing the problem to a finitary version of it. Using our approach, it follows immediately that every coarsely separating family of subsets of a $d$-dimensional Euclidean building or a product of $d$ geodesic, geodesically complete metric spaces has asymptotic dimension at least $d-1$. As a corollary, we obtain obstructions to coarse embeddings of Euclidean spaces into certain fundamental groups of graphs of groups.
\end{abstract}

\section{Introduction}
A subset $A$ of a topological space $X$ is said to \emph{separate} $X$ if the complement $X \setminus A$ consists of more than one connected component. The classical Jordan curve theorem states that every simple closed curve separates the Euclidean plane, and that the complement consists of exactly two connected components. The notion of ``coarse separation’' is a coarse-geometric analogue of topological separation (cf. Section \ref{subsection:coarsesep} for the definition). Suppose $X$ is a connected metric space. If a subset $A$ coarsely separates $X$, then there exists $L>0$ such that $X \setminus N_L(A)$ has more than one component. Moreover, for all $R>0$, there exist points $x,y \in X$ lying in distinct components of $X \setminus N_L(A)$ such that $d(x,A)>R$ and $d(y,A)>R$.
\par
In geometric group theory, coarse separation arises quite naturally. If a finitely generated group $G$ (equipped with some word-metric) splits over a subgroup $C$ as an amalgam or an HNN extension then $C$ coarsely separates $G$. Stallings' theorem about ends states that the Cayley graph of a finitely generated group $G$ is coarsely separated by a bounded subset  if and only if $G$ splits as an amalgam or an HNN extension over one of its finite subgroups. Coarse separation is a key ingredient in the proof of quasi-isometric rigidity for certain groups \cite{10.4310/jdg/1214459217},\cite{MosherSageevWhyte2011}.  
Tessera–Bensaid–Genevois \cite{Bensaid2024CoarseSA} define a notion of coarse separation for families of subspaces of a metric space. They show that certain spaces cannot be separated by families of subspaces of subexponential growth. Examples include symmetric spaces of non-compact type other than $\mathbb{H}^2_{\mathbb{R}}$, horocyclic products of trees, and some of Bourdon’s buildings.
\par
Asymptotic dimension, denoted by $\operatorname{asdim}$, is a large-scale analogue of Lebesgue covering dimension for metric spaces. Roughly speaking, a metric space $X$ has asymptotic dimension at most $n$ if, at every sufficiently large scale, it can be decomposed into uniformly bounded pieces in such a way that no point lies in more than $(n+1)$ pieces and two pieces that do not intersect are sufficiently far apart. This invariant is coarse in nature, meaning it is preserved under quasi-isometries. Since its inception by Gromov in the early 1990's it has become a central tool in geometric group theory. For a general introduction to the topic we refer the reader to \cite{BellDranishnikov2008}.       The notion of asymptotic dimension has been extended to families of metric spaces. Various families of graphs have been studied in this context \cite{Bonamy2023-es},\cite{FujiwaraPapasoglu2006},\cite{Liu2021Asymptotic},\cite{DELABIE20181036},\cite{DVORAK2025103631}.
\par
The study of the asymptotic dimension of coarsely separating subsets is relatively recent. Papasoglu and Delzant \cite{DelzantPapasoglu2010} constructed hyperbolic groups of arbitrarily high asymptotic dimension that do not split over any proper subgroup and yet have coarsely separating subsets of asymptotic dimension one. 
In a recent preprint\cite{Tselekidis2024CoarseSeparation}, Tselekidis proved that any vertex transitive graph of asymptotic dimension $n$ admits a coarsely separating subset of dimension strictly less than $n$. In the same article, Tselekidis proposes the following conjecture:
\begin{center}
    \emph{$\mathbb{R}^d$ cannot be coarsely separated by a subspace of asymptotic dimension strictly less than $d - 1$.}
\end{center}
The author proved this conjecture  \cite{Patil2025CoarseSO}, and in fact established it for a much larger class of spaces, namely, coarse $\mathrm{PD}(d)$ spaces:
\begin{theorem}\label{maintheorem}
Let $X$ be a coarse $\mathrm{PD}(d)$ space, and let $A \subseteq X$ be a subset that coarsely separates $X$. Then the asymptotic dimension of $A$ is at least $d-1$.
\end{theorem}
For a definition of coarse $\mathrm{PD}(d)$ spaces see \cite{BanerjeeOkun2025}. 
One of the main tools in the proof of Theorem~\ref{maintheorem} is the coarse Alexander duality theorem of Banerjee–Okun \cite{BanerjeeOkun2025}. This result is formulated in the language of Roe’s coarse cohomology, and its use renders the proof of Theorem~\ref{maintheorem} somewhat opaque. It is also natural to ask whether there are spaces which are not coarse $\mathrm{PD}(d)$ but satisfy the conclusion of Theorem \ref{maintheorem}. 
\par
In this article, we give a more illuminating proof of Theorem~\ref{maintheorem} in the special case $X = \mathbb{R}^d$, and we extend the conclusion to a broader class of spaces, namely those satisfying property $QF_d$ i.e., spaces having a rich collection of $d$-dimensional quasiflats (cf. Section \ref{QF} for the definition). We use only ordinary (simplicial/singular) cohomology in our proof. 
 
In addition to $\mathbb{R}^d$, $QF_d$ is satisfied by any Euclidean building of dimension $d$ and any $\ell_\infty$ product of $d$ geodesic, geodesically complete metric spaces.  We believe that the proof strategy should be applicable to an even larger collection of spaces such as the $d$-dimensional rank one symmetric spaces and $d$-dimensional hyperbolic buildings. We also work in the more general framework, introduced by Bensaid-Tessera–Genevois \cite{Bensaid2024CoarseSA}, in which they consider families of separating subsets. Our main result is the following:

\begin{theorem}\label{main}
    Let $X$ be a metric space that satisfies property $\textrm{QF}_{d}$. Let $\mathcal{A}$ be a family of subsets of $X$ such that $\mathcal{A}$ coarsely separates $X$. Then, $\operatorname{asdim}(\mathcal{A})\geq d-1$.    
\end{theorem}
\begin{corollary}
    If $X$ is the Euclidean building of dimension $d$, or a product $X=\Pi_{i=1}^{n}X_i$ of $d$ geodesic, geodesically complete metric spaces then any family which coarsely separates $X$ has asymptotic dimension at least $d-1$.    
\end{corollary}
To prove this result, we first reduce the problem to the following finitary version:

\begin{theorem}\label{mainresult}
Let $(L_n)_n$ and $(b_n)_n$ be sequences of positive real numbers tending to infinity.  
For each $n$, let $C_n=[0,L_n]^d\subset\mathbb{R}^d$ be the $d$-dimensional cube equipped with the $\ell_\infty$-metric. For each $n$, let $A_n\subseteq C_n$ such that 
\begin{enumerate}
\item $C_n\setminus A_n$ has more than one connected component;
\item there exist points $x_n,y_n\in C_n$ lying in different connected components of $C_n\setminus A_n$ such that
$$
d(x_n,A_n\cup\partial C_n)>b_n, \qquad d(y_n,A_n\cup\partial C_n)>b_n .
$$
\end{enumerate}
Let $\mathcal{A}=\{A_n\}_n$. Then $\operatorname{asdim}(\mathcal{A})\ge d-1$.
\end{theorem}

\begin{figure}[H]
    \centering
    \includegraphics[width=170pt]{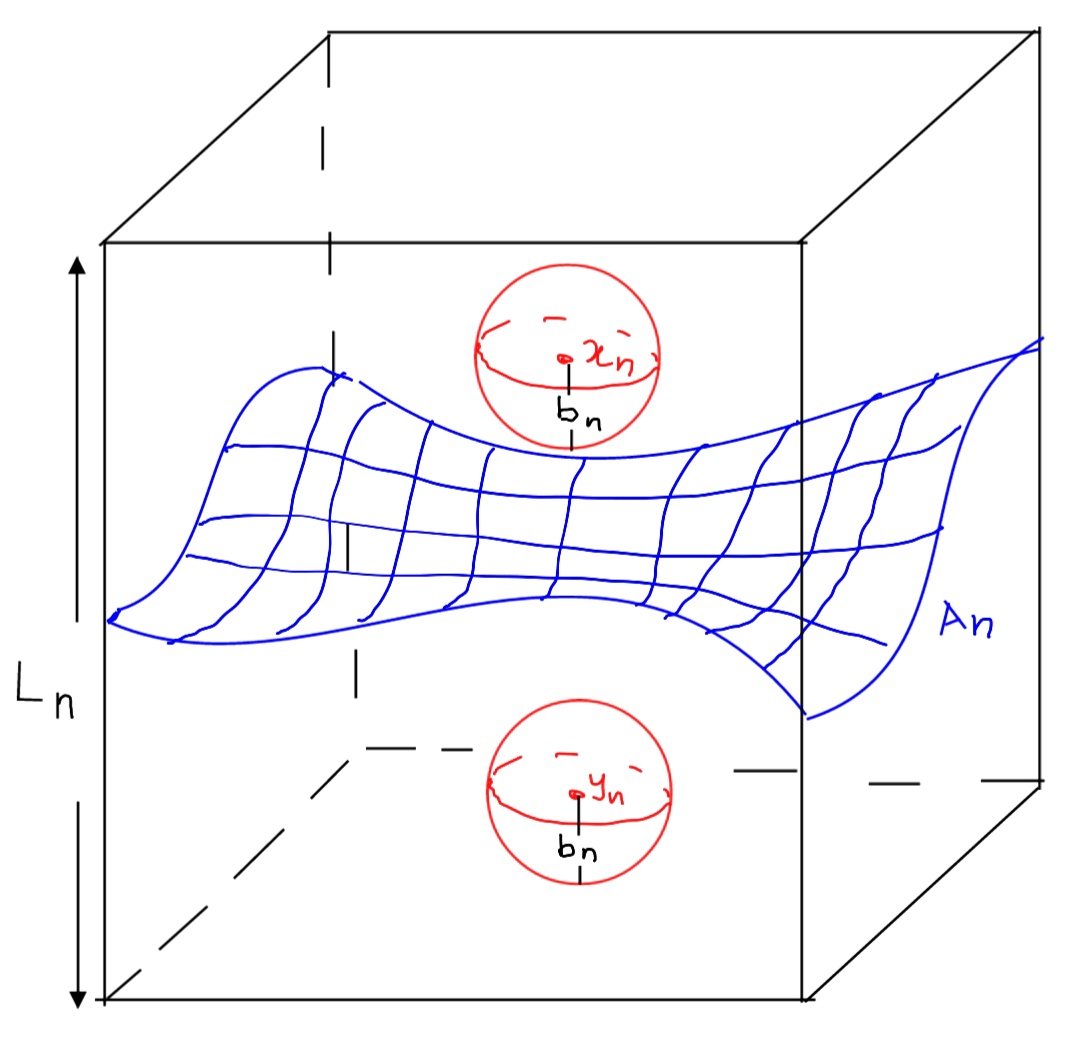}
    \caption{}
\end{figure}
We believe that this result is of independent interest. The proof of Theorem \ref{mainresult} uses Theorem \ref{poincarelefschetz}, a variant of the classical Alexander duality theorem for manifolds with boundary. Another main ingredient of the proof is Lemma \ref{lemmahomology} which provides a homological factorization property for refinement maps between nerves of covers of families with uniformly bounded asymptotic dimension. 
\par
Finally, we give an application towards ruling out certain coarse embeddings. A coarse embedding between finitely generated groups is a map that preserves large-scale geometry while ignoring small-scale distortion. Concretely, it sends pairs of group elements that are far apart (with respect to word metrics) to points that remain far apart, up to uniform control functions (cf. section \ref{section:quasiisometry} for the definition). 
So far, most of the focus in the study of coarse embeddings has been on invariants that are coarse monotone, i.e., invariants $\mathbb{I}$ equipped with some partial order $(\mathbb{I},\leq)$ and such that if a  space $X$ coarsely embeds in another space $Y$ then $\mathbb{I}(X) \leq \mathbb{I}(Y)$. Examples of such invariants include growth, asymptotic dimension, separation profiles, $L_p$-compression exponents and Poincare profiles. Bensaid-Tessera-Genevois \cite{Bensaid2024CoarseSA} use coarse separation to study coarse embeddings onto two classes of finitely generated groups: amalgamated products and wreath products.
 Inspired by their approach, in section \ref{applications} we give an application of Theorem \ref{main} towards understanding coarse embeddings of spaces with property $\textrm{QF}_d$ into fundamental groups of certain graphs of groups.
 \\
\paragraph{\textbf{Conventions.}}
\begin{itemize}
    \item Throughout this paper, all homology and cohomology groups are taken with coefficients in $\mathbb{F}_2$, and `$\operatorname{rank}$' denotes the vector space dimension over $\mathbb{F}_2$. This is done mainly to avoid issues pertaining to orientation.  
    \item In Section \ref{section:covers}, we work with reduced homology, whereas ordinary homology and cohomology are used throughout the rest of the paper. This makes the proof of Lemma \ref{lemthree} conceptually easier, and this distinction causes no difficulties since Lemma \ref{lemthree} concerns dimensions greater than zero.

    \item We endow $\mathbb{R}^d$ with the $l_\infty$ metric. As $l_1, l_2$
 and $l_\infty$ metrics on $\mathbb{R}^d$ are quasi-isometric to each other all the statements hold when one replaces the $l_\infty$ metric by either the $l_1$ or $l_2$ metric. 
 \item Let $X$ be a metric space and let $A $ be a subset of $X$. For $L>0$ we denote the closed $L$-neighborhood of $A$ by $N_L(A)$. 
 \end{itemize}

\section{Preliminaries}

\subsection{Coarse separation of metric spaces}\label{subsection:coarsesep}

We recall what it means for a family $\mathcal{S}$ of subsets to coarsely separate a metric space $X$. 

\begin{definition}[coarse separation]\cite{Bensaid2024CoarseSA}
Let $X$ be a connected metric space and let $\mathcal{S}$ be a family of subsets
of $X$. We say that $\mathcal{S}$ coarsely separates $X$ if:
\begin{enumerate}
    \item there exists $L>0$ such that $X\setminus N_{L}(S)$ contains more than one path-component for all $S\in \mathcal{S}$. 
    \item for all $D>0$, there exist $S\in \mathcal{S}$ and points  $x,y\in X$ such that $x$ and $y$ lie in distinct components of $X\setminus N_L(S)$ such that $d(x,S)>D$ and $d(y,S)>D$. 
\end{enumerate}
\end{definition}

\begin{remark}\label{remark:zero}
    The constant $L$ in the above definition can be assumed to be zero by replacing the family $\mathcal{S}$ by $\mathcal{S}^{+L}=\{N_L(U)|U\in \mathcal{S}\}$. The family $\mathcal{S}^{+L}$ has the same asymptotic dimension as $\mathcal{S}$. 
\end{remark}

\subsection{Quasi-isometries and coarse embeddings}\label{section:quasiisometry}
\begin{definition}
Let $X$ and $Y$ be metric spaces and $K\geq 1$. 
A map $f:X\rightarrow Y$ is said to be a $K$-\textit{quasi-isometric embedding} if for all $x,y\in X$, 
$$ K^{-1}d(x,y)-K\leq d(f(x),f(y))\leq Kd(x,y)+K.$$
$f$ is said to be a $K$-\textit{quasi-isometry} if, in addition to being a $K$-quasi-isometric embedding, $f$ satisfies the following condition:
\begin{center}
    for all $y \in Y$, there exists a $y'\in \operatorname{Im}(f)$ such that $d(y,y')\leq K$.  
\end{center}
\end{definition}
Let $\mathcal{A}$ and $\mathcal{B}$ be two families of metric spaces. We say that $\mathcal{A}$ \emph{quasi-isometrically embeds} into $\mathcal{B}$ if for some fixed $K\geq 1$ every $A\in \mathcal{A}$ admits a $K$-quasi-isometric embedding $f:A\rightarrow B$ into some $B\in \mathcal{B}$.  
\begin{definition}[quasi-inverse]
  Let $f:X\rightarrow Y$ be a $K$-quasi-isometry. Let $K'\geq 1$. A $K'$-quasi-isometry  
$g:Y\rightarrow X$ is said to be a $K'$-quasi-inverse of $f$ if $d(f\circ g, id_Y )\leq K'$ and $d(g\circ f, id_X)\leq K'$.  
\end{definition}
 The following lemma shows that a quasi-inverse always exists. 
\begin{lemma}
   Let $X,Y$ be metric spaces. Let $K\geq 1$ and $K'=\max\{2,3K^2\}$. Given a $K$-quasi-isometry $f:X\rightarrow Y$, there exists a $K'$ quasi-inverse $g:Y \rightarrow X$ of $f$.
\end{lemma}
\begin{proof}
For each $y\in Y$ there exists a $y'\in Im(f)$ such that $d(y,y')\leq K$. Let $t:Y\rightarrow Im(f)$ be a function such that $t(y)\in Im(f)$ for all $y\in Y$ and $d(y,t(y'))\leq K$. Let $s: Im(f)\rightarrow X$ be a function such that $s(y')\in f^{-1}(y')$ for all $y'\in Im(f)$. Define $ g:=s\circ t$.
\\
Let $y_1,y_2\in Y$. Then, $d(t(y_1),y_1)\leq K$ and $d(t(y_2),y_2)\leq K$. Thus, $$d(y_1,y_2)-2K\leq d(t(y_1),t(y_2))\leq d(y_1,y_2)+2K$$ Also, by definition, $f(s\circ t(y_1))=t(y_1)$ and $f(s\circ t(y_2))=t(y_2)$.\\
As 
$f$ is a $K$-quasi-isometry, 
$$K^{-1}d(s\circ t(y_1),s\circ t(y_2))-K\leq d(t(y_1),t(y_2))\leq K d(s\circ t(y_1),s\circ t(y_2))+K.$$
We have the following,
\begin{equation*}\label{one}
   d(s\circ t(y_1), s\circ t(y_2))\leq Kd(t(y_1),t(y_2))+K^2 \leq Kd(y_1,y_2)+3K^2\leq K'd(y_1,y_2)+K'
\end{equation*}
For the lower bound, 
\begin{equation*}
    d(s\circ t(y_1), s\circ t(y_2))\geq \frac{d(t(y_1),t(y_2))-K}{K}  \geq \frac{d(y_1,y_2)-3K}{K}\geq (K')^{-1}d(y_1,y_2)-K'
\end{equation*}
It remains to show that $d(f\circ g, id_X )\leq K'$ and $d(g\circ f, id_Y)\leq K'$.
Let $x \in X$ and $y=f(x)$. Let $y'=t(y)$ and $x'=s(y')=f\circ g(x)$. Then $x\in f^{-1}(y)$ and $x'\in f^{-1}(y')$. Therefore, 
$$d(x,x')\leq \frac{d(y,y')+K}{K}\leq \frac{2K}{K}.$$
Now let $y\in Y$ and let $x=g(y)=s\circ t(y)$. Then, as $s(t(y)) \in f^{-1}(t(y))$, $f\circ g(y)=t(y)$. Thus,
$d(y,f\circ g(y))\leq K$. 
\end{proof}
We recall the definition of a coarse embedding. 
\begin{definition}[coarse embedding]
Let $(X,d_X)$ and $(Y,d_Y)$ be metric spaces.
A map $f \colon X \to Y$ is called a coarse embedding if there exist non-decreasing functions
$\rho_-, \rho_+ \colon [0,\infty) \to [0,\infty)$ such that $\rho_-(r) \rightarrow \infty$ as $r\rightarrow \infty$  and
$$\rho_-\bigl(d_X(x,x')\bigr)\;\le\;d_Y\bigl(f(x),f(x')\bigr)
\;\le\;\rho_+\bigl(d_X(x,x')\bigr)\quad \text{for all } x,x' \in X$$

\end{definition}

Let $\mathcal{A}$ and $\mathcal{B}$ be two families of metric spaces. We say that $\mathcal{A}$ \emph{coarsely embeds} into $\mathcal{B}$ if for some fixed $\rho+,\rho_-$, such that $\rho_-(r), \rho_+(r) \rightarrow \infty$ as $r\rightarrow \infty$, every $A\in \mathcal{A}$ admits a mapping $f:A\rightarrow B$ into some $B\in \mathcal{B}$
such that $$\rho_-\bigl(d_X(x,x')\bigr)\;\le\;d_Y\bigl(f(x),f(x')\bigr)
\;\le\;\rho_+\bigl(d_X(x,x')\bigr)\quad \text{for all } x,x' \in A.$$

\begin{definition}[coarse equivalence]
Let $(X,d_X)$ and $(Y,d_Y)$ be metric spaces.
A map $f \colon X \to Y$ is called a coarse equivalence if it is a coarse embedding and there exists a constant $C \geq 0$ such that every point of $Y$ lies within distance $C$ of the image of $f$, that is,
$$
d_Y(y,f(X))\le C\qquad\text{for all }y\in Y.
$$
Equivalently, the $C$--neighbourhood of $f(X)$ is equal to $Y$.
\end{definition}
\subsection{Asymptotic dimension}
Let $(X,d)$ be a metric space. A cover $\mathcal{U}$ of $X$ is said to have \textit{multiplicity} less or equal to than $k$ if every $x\in X$ is contained in at most $k$ elements of $\mathcal{U}$. The \textit{Lebesgue number} of a cover $\mathcal{U}$, denoted by $\mathcal{L}(\mathcal{U})$ is the supremum over all positive real numbers $\lambda$ such that for any set $A\subseteq X$ with diameter $\operatorname{diam}(A)\leq \lambda$ there exists $U\in \mathcal{U}$ such that $A \subseteq  U$. We say that a cover $\mathcal{U}$ is a \textit{refinement} of $\mathcal{V}$ if for every $U\in \mathcal{U}$ there exists an element $V\in \mathcal{V}$ such that $U\subseteq  V$. We denote this by $\mathcal{U} \preccurlyeq \mathcal{V}$. A cover $\mathcal{U}$ is said to be $R$-bounded if $\sup_{U\in \mathcal{U}}\{\operatorname{diam}(U)\}\leq R$. 
Let $r>0$. The \textit{mesh} of a uniformly bounded cover $\mathcal{U}$ is defined as  $\textrm{mesh}(\mathcal{U})=\sup_{U\in \mathcal{U}}\{\textrm{diam}(U)|U\in \mathcal{U}\}$. A collection $\mathcal{W}$ of subsets of $X$ is said to be $r$-\emph{disjoint} if for any two elements $W,W'\in \mathcal{W}$, $d(W,W')\geq r$.
 
\begin{definition}[Asymptotic dimension]
   Let $\mathcal{X}$ be a family of metric spaces. We say that $D_{\mathcal{X}}:\mathbb{R}_{+}\rightarrow \mathbb{R}_{+}$ is an $n$-dimensional control function if for any $r>0$ and for any $X\in \mathcal{X}$, $X$ has a $D_{\mathcal{X}}(r)$-bounded cover $\mathcal{U}=\bigcup_{i=1}^{n+1}\mathcal{U}_{i}$ such that each $\mathcal{U}_{i}$ is $r$-disjoint. The asymptotic dimension of $\mathcal{X}$, denoted by  $\operatorname{asdim}(\mathcal{X})$, is the least integer $n$ such that $\mathcal{X}$ admits an $n$-dimensional control function.  
\end{definition}

The asymptotic dimension of a metric space $X$ is just the asymptotic dimension of the singleton family $\{X\}$. 

\begin{lemma}\label{asdiminvariance}
    Let $\mathcal{A}$ and $\mathcal{B}$ be families of metric spaces such that $\mathcal{A}$ coarsely embeds in $\mathcal{B}$. Then $\operatorname{asdim}(\mathcal{A}) \leq \operatorname{asdim}(\mathcal{B})$. 
\end{lemma}

\begin{proof}
Let $ \operatorname{asdim}(\mathcal{B})=k<\infty$ and let $D$ be a $k$-dimensional control function for $\mathcal{B}$. Let $\rho+,\rho_-$ be functions  such that $\rho_-(r), \rho_+(r) \rightarrow \infty$ as $r\rightarrow \infty$, and every $A\in \mathcal{A}$ admits a mapping $f:A\rightarrow B$ into some $B\in \mathcal{B}$
such that $$\rho_-\bigl(d_X(x,x')\bigr)\;\le\;d_Y\bigl(f(x),f(x')\bigr)
\;\le\;\rho_+\bigl(d_X(x,x')\bigr)\quad \text{for all } x,x' \in A.$$ Given $R>0$, let $\mathcal{U}=\bigcup_{i=1}^{k+1}\mathcal{U}_{i}$ be a cover of $B$ such that $ \mathcal{U}$ is $D(\rho_+(R))$-bounded and each $\mathcal{U}_{i}$ is $(\rho_+(R))$-disjoint. Let $\mathcal{V}=\{f^{-1}(U)|U\in \mathcal{U}\}$ and let $\mathcal{V}_i=\{f^{-1}(U)|U\in \mathcal{U}_i\}$ be so that $\mathcal{V}=\bigcup_{i=1}^{k+1}\mathcal{V}_{i}$. Then, each $\mathcal{V}_i$ is $R$-disjoint. Let $N>0$ such that $\rho_-(N)>D(R)$. We claim that  $ \mathcal{V}$ is $N$-bounded.  If $ \operatorname{diam}(f^{-1}(U_i))> N$ then $\operatorname{diam}(U_i)> D(R)$, a contradiction.
\end{proof}

\begin{lemma}\label{Lebesguenumber}
    Let $\mathcal{A}$ be a family of metric spaces such that $\operatorname{asdim}(\mathcal{A})\leq k$. Then, there exists a function $g_{\mathcal{A}}:\mathbb{R}_{+}\rightarrow \mathbb{R}_{+}$ such that for each $A\in \mathcal{A}$, there exists a $g_{\mathcal{A}}(\lambda)$-bounded cover  $\mathcal{U}$ of $A$ with Lebesgue number $\geq \lambda$ and multiplicity $\leq k+1$. 
\end{lemma}

\begin{proof} 
Let $D_{\mathcal{\mathcal{A}}}:\mathbb{R}_{+}\rightarrow \mathbb{R}_{+}$ be a $k$-dimensional control function for $\mathcal{A}$. We prove that the function $g_{\mathcal{A}}(x)=D_{\mathcal{A}}(5x)$ satisfies the conclusion of the statement. For a given $A\in \mathcal{A}$ let $\mathcal{V}$ be a $D_{\mathcal{A}}(5\lambda)$-bounded cover of $A$ that can be expressed as a union of $k+1$ families $\mathcal{V}_i$ , $i=0,1,\dots,k$ such that each $\mathcal{V}_i$ is $5\lambda$-disjoint. Let $\mathcal{W}$ denote the cover $\{N_{2\lambda}(V)|V\in \mathcal{V}\}$. We first show that $\mathcal{W}$ has multiplicity $\leq k+1$. For a fixed $i\in \{0,1,\dots,k\}$, any $x \in X$ is contained in $N_{2\lambda}(V)$ for at most one $V\in \mathcal{V}_i$, As the collection $\mathcal{V}_i$  is $5\lambda$-disjoint. It follows that any $x\in X$ is contained in at most $k+1$ elements of $\mathcal{W}$. Let $U\subseteq X$ be a subset of diameter $\leq \lambda$ and let $x\in U$. Let $V\in \mathcal{V}$ such that $x\in V$. Then, $U\subseteq B_{\lambda}(x)\subseteq N_{2\lambda}(V)$. Consequently, $\mathcal{L}(\mathcal{W})\geq \lambda$. 
\end{proof}
\subsection{Simplicial complexes, barycentric subdivisions}
Simplicial complexes can be defined abstractly as families of subsets of a given set closed under taking subsets or as topological spaces obtained by gluing simplices along their faces. We work with both points of view in this paper and switch between the two. For an introduction to abstract simplicial complexes we refer the reader to \cite{Kozlov}. 

\begin{definition}
A \emph{simplicial complex} $K$ consists of a vertex set $K^{(0)}$ together with a collection of finite subsets of $K^{(0)}$, called \emph{simplices}, satisfying the following condition:
\begin{center}
      If $\sigma \in K$ and $\tau \subseteq  \sigma$, then $\tau \in K$.
\end{center}

An $n$-simplex is a simplex containing exactly $(n+1)$ vertices.

\end{definition}
A subset $L$ of $K$ that is a simplicial complex in its own right is called a \emph{subcomplex} of $K$. A \textit{pair} $(K,L)$ of complexes consists of a simplicial complex $K$ and a subcomplex $L$.  
As is customary, we use square brackets to denote simplices, $\sigma=[U_0,\dots ,U_n]$.
\begin{definition}[Simplicial map]
Let $K$ and $L$ be simplicial complexes. A \emph{simplicial map} $f \colon K \to L$ is a map from the vertex set of $K$ to the vertex set of $L$ such that for every simplex $\sigma = [v_0,\dots,v_n]$ of $K$, 
$f(\sigma)=[f(v_0),\dots,f(v_n)]$
is a simplex of $L$.
\end{definition}
\begin{definition}[Contiguous maps]
Let $(K,A)$ and $(L,B)$ be pairs of simplicial complexes, and let
$f,g \colon (K,A) \to (L,B)$
be simplicial maps. We say that $f$ and $g$ are contiguous if, for every simplex $\sigma = [v_0,\dots,v_k]$ of $K$, the set of vertices
$\{ f(v_0),\dots,f(v_k), g(v_0),\dots,g(v_k) \}$
spans a simplex $\sigma'$ in $L$, and moreover, if $\sigma$ lies in $A$, then $\sigma'$ lies in $B$.
\end{definition}
We will need the following standard result which says that contiguous simplicial maps induce the same homomorphism on relative cohomology.

\begin{lemma}\cite[Theorem 12.6]{Munkres84}\label{simplicialhomotopy}
    Let $f,g:(K,A)\rightarrow (L,B)$ be contiguous maps. Then, the maps $f^{\ast}, g^{\ast}:H^{i}(L,B)\rightarrow H^{i}(K,A)$ induced by $f$ and $g$ on the relative cohomology groups are equal for all $i$. 
\end{lemma}
\begin{definition}
   A geometric simplicial complex $K$ is a finite collection of geometric simplices  such
that 
\begin{enumerate}
    \item for any simplex $\sigma \in K$ if $\tau$ is a face of $\sigma$ then $\tau \in K$
    \item for any two simplices $\sigma, \sigma' \in K$ if $\sigma\cap\sigma'\neq \emptyset$ then $\sigma\cap\sigma'$ is a common face of both $\sigma$ and $\sigma'$.
\end{enumerate}
\end{definition}
Given an abstract simplicial complex there is a natural way of associating a geometric simplicial complex to it: 
\begin{definition}[Geometric realization]
    Let $K$ be an abstract simplicial complex with vertex set $V$. The geometric realization $|K|$ is a topological space obtained by replacing each abstract simplex of $K$ with a geometric simplex and gluing them together according to the face relations in $K$.

The construction proceeds as follows.

\begin{enumerate}
    \item For each vertex $v \in V$, let $e_v$ denote the corresponding standard basis vector in the vector space
    $\mathbb{R}^{(V)}$. 

    \item For every simplex $\sigma = \{v_0,\dots,v_n\} \in X$, define the geometric simplex associated to $\sigma$ by
    $
    |\sigma|
    =
    \left\{
        \sum_{i=0}^n t_i e_{v_i}
        \;\middle|\;
        t_i \ge 0,\;
        \sum_{i=0}^n t_i = 1
    \right\}.
    $
    This is the convex hull of the vertices $e_{v_0},\dots,e_{v_n}$.

    \item Define the geometric realization of $X$ as the union of all such simplices:
    $
    |K|
    =
    \bigcup_{\sigma \in K} |\sigma|.
    $

    \item Equip $|K|$ with the subspace topology inherited from $\mathbb{R}^{(V)}$.
\end{enumerate}

Equivalently, one may describe $|K|$ as the set
$$
|K|
=
\left\{
(t_v)_{v \in V}
\in \mathbb{R}^{(V)}
\;\middle|\;
t_v \ge 0,\;
\sum_{v \in V} t_v = 1,\;
\{v \in V : t_v \neq 0\} \in K
\right\}.
$$
\end{definition}
The coefficients $t_v$ are called the barycentric coordinates of the point $(t_v)_{v\in V}$.
\begin{definition}[Star]
   Let $K$ be a simplicial complex and let $L\subseteq K$ be a subcomplex. The star $\operatorname{St}_K(L)$ denoted by $\operatorname{St}_K(L)$, is defined by \\
   $$
\operatorname{St}_K(L)
=
\bigcup \left\{ \sigma
\;\middle|\;
\sigma \in K,\ \sigma \cap L \neq \varnothing
\right\}.
$$
   When the ambient complex is clear from the context we drop the subscript and refer the star of a subcomplex simply by $\operatorname{St}(\cdot)$. 
\end{definition}

\begin{definition}[Open Star]
Let $K$ be a simplicial complex and let $L \subseteq  K$ be a subcomplex. The open star of $L$ in $K$, denoted by $\operatorname{st}_K(L)$, is defined by
$$
\operatorname{st}_K(L)
=
\bigcup \left\{ \operatorname{int}(\sigma)
\;\middle|\;
\sigma \in K,\ \sigma \cap L \neq \varnothing
\right\},
$$
where $\operatorname{int}(\sigma)$ denotes the relative interior of the simplex $\sigma$.

Equivalently, $\operatorname{st}_K(L)$ is the union of the relative interiors of all simplices of $K$ having a face in $L$.\\
When the ambient complex is clear from the context we drop the subscript and refer the open star simply by $\operatorname{st}(\cdot)$. 
\end{definition}

A collection of $(k+1)$ vectors $v_0,v_1,\dots ,v_k$ in a vector space  is said to be in \textit{general position} if the set $\{v_i-v_0\}_{i=1}^{k}$ is linearly independent. Given a collection of $(k+1)$ vectors $v_0,v_1,\dots ,v_k$ in general position the (geometric) $k$-simplex spanned by them is  the set $$[v_0,v_1,\dots ,v_k]=\left\{\sum_{i=0}^{k} t_i v_i
\;\middle|\;
t_i \ge 0,\;
\sum_{i=0}^{k} t_i = 1
\right\}. $$

 Given a simplex $\sigma=[v_0,v_1,\dots ,v_k]\subseteq \mathbb{R}^N$ the \textit{barycenter} $b_\sigma$ is defined to be the point $ \frac{1}{k+1}(v_0+v_1+\dots+v_k)$. 
\begin{definition}[geometric barycentric subdivision]
   Let $X\subseteq \mathbb{R}^N$ be a geometric simplicial complex. The barycentric subdivision of $X$ is defined to be the flag simplicial complex such that: 
   \begin{enumerate}
       \item The $0$-skeleton of $Bd(X)$ is the set obtained by taking barycenters of simplices of $X$: 
       $$Bd(X)^{(0)}=\{b_{\sigma}|\sigma \in X \}$$
       \item The $1$-skeleton is given by,  $$Bd(X)^{(1)}=\{[b_{\sigma}, b_{\sigma'}]|\textrm{ either }\sigma \subseteq \sigma' \textrm{ or } \sigma'\subseteq \sigma\}$$   
   \end{enumerate}
\end{definition}

The following lemma provides an upper bound on the diameter of the simplices in a the barycentric subdivision of a given geometric simplex:

\begin{lemma}\cite[pg. 120]{hatcher}\label{lemma:bddiameter} Let $x_0,\dots,x_d$ be a collection of $(d+1)$ vectors in $\mathbb{R}^d$. 
    The diameter of each simplex of the barycentric subdivision of the simplex  $[x_0,x_1,\dots, x_d]$ is at most $\frac{d}{d+1}\max\{d(x_i,x_j)\}_{0\leq i,j\leq d}$. 
\end{lemma}

\subsection{Nerves, refinement maps}
In this subsection, we recall the construction of nerves associated with covers of metric spaces and introduce refinement maps between them.
\begin{definition}[Nerve of a cover]
 Let $X$ be a metric space and let  $\mathcal{U}$ be a cover of $X$ by uniformly bounded sets. The nerve of $\mathcal{U}$ ,denoted by $N(\mathcal{U})$, is the simplicial complex defined as follows:
 \begin{enumerate}
     \item The zero-skeleton $N(\mathcal{U})^{(0)}$ is $\mathcal{U}$. 
      \item An $(n+1)$-tuple $[U_{0},U_{1}, \dots,U_{n}]$ spans an $n$-simplex if and only if the intersection $U_{0}\cap  U_{1}\cap\dots\cap U_{n}$ is non-empty. 
 \end{enumerate}
\end{definition}

\begin{definition}[Refinement map]
    Let $\mathcal{U}$ and $\mathcal{V}$ be two covers of $X$. We say that $ \mathcal{U}$ is a refinement of $\mathcal{V}$ if for every $ U\in \mathcal{U}$ there exists an element $V\in \mathcal{V}$ such that $U\subseteq V$. We denote this by $\mathcal{U}\preccurlyeq \mathcal{V}$. 
\end{definition}

Let $\mathcal{U}$,$\mathcal{V}$ be two covers of $X$ such that $\mathcal{U}\preccurlyeq \mathcal{V}$. A \textit{refinement map} $p_{\mathcal{U}}^{\mathcal{V}}$ is a map that $p_{\mathcal{U}}^{\mathcal{V}}:\mathcal{U}\rightarrow \mathcal{V}$ that sends a given $U\in \mathcal{U}$ to an element $V\in \mathcal{V}$ such that $U\subseteq V$. Any refinement map extends to a simplicial map $N(\mathcal{U})\rightarrow N(\mathcal{V})$. Slightly abusing notation we denote this extension to the corresponding nerves again by $p_{\mathcal{U}}^{\mathcal{V}}$. 

    \begin{lemma}\label{relativehomotopy}
   Let $X$ be a metric space and  let $\mathcal{U}$ and $\mathcal{V}$ be covers of $X$ such that $\mathcal{U}\preccurlyeq \mathcal{V}$. Let $P$ and $Q$ be subcomplexes of $N(\mathcal{U})$ and $N(\mathcal{V})$ such that any refinement map $p_{\mathcal{U}}^{\mathcal{V}}$ maps $P$ into $Q$. Assume further that for any two refinement maps $p_{\mathcal{U}}^{\mathcal{V}},q_{\mathcal{U}}^{\mathcal{V}}$ and for any simplex $ \sigma=[U_0,U_1,\dots,U_n]\in P$ the simplex spanned by $\{p_{\mathcal{U}}^{\mathcal{V}}(U_i)\}_{i=0}^{n}\cup\{q_{\mathcal{U}}^{\mathcal{V}}(U_i)\}_{i=0}^{n}$ lies in $Q$. Then, any two refinement maps $p_{\mathcal{U}}^{\mathcal{V}},q_{\mathcal{U}}^{\mathcal{V}}:(N(\mathcal{U}),P) \rightarrow (N(\mathcal{V}),Q)$  are contiguous.
\end{lemma}

\begin{remark}
We allow the possibility that $P$ and $Q$ are empty in the above statement. 
\end{remark}

\begin{proof}
Let $ \sigma=[U_0,U_1,\dots ,U_k]$ be a $k$-simplex, then $\cap_{i=1}^{k} U_i \neq \emptyset$. Let $p_{\mathcal{U}}^{\mathcal{V}}(U_i)=V_i$ and $q_{\mathcal{U}}^{\mathcal{V}}(U_i)=W_i$. We first show that the set  $\{V_i\}_{i=0}^{k}\cup \{W_i\}_{i=0}^{k}$ spans a simplex $\sigma'$. By definition, $U_i\subseteq V_i$ and $U_i\subseteq W_i$, for all $0\leq i\leq k$. As a result, $\cap_{i=0}^{k} U_i \subseteq \cap_{i=0}^{k} V_i$ as well as $\cap_{i=0}^{k} U_i \subseteq \cap_{i=0}^{k} W_i$. It follows that $ (\cap_{i=0}^{k} V_i) \cap (\cap_{i=0}^{k}W_i) \neq \emptyset$ , and hence the set $\{V_i\}_{i=0}^{k}\cup \{W_i\}_{i=0}^{k}$ spans a simplex. Now, if $\sigma$ is a simplex in $P$ then, by assumption, the simplex spanned by $\{V_i\}_{i=0}^{k}\cup \{W_i\}_{i=0}^{k}$ is contained in $Q$. It follows that $\sigma'\in Q$.  
\end{proof}

\begin{lemma}\label{lemmahomology}
Let $\mathcal{A}$ be a family of metric spaces such that $\operatorname{asdim}(\mathcal{A}) \le k$. Then there exists a function $g \colon \mathbb{R}_{+} \to \mathbb{R}_{+}$ such that for any $A \in \mathcal{A}$ and any covers $\mathcal{U}, \mathcal{V}$ of $A$ with $\operatorname{mesh}(\mathcal{U}) \le \lambda$ and $\mathcal{L}(\mathcal{V}) \ge g(\lambda)$, the following conditions are satisfied:

\begin{enumerate}
    \item There exists a cover $\mathcal{W}$ of $A$ with multiplicity at most $(k+1)$ such that
    \[
        \mathcal{U} \preccurlyeq \mathcal{W} \preccurlyeq \mathcal{V}.
    \]

    \item Any refinement map
    \[
        p_{\mathcal{U}}^{\mathcal{V}} \colon N(\mathcal{U}) \longrightarrow N(\mathcal{V})
    \]
    factors, up to contiguity, through the $k$-dimensional complex $N(\mathcal{W})$:
    \[
    \begin{tikzcd}
        & N(\mathcal{W}) \arrow{dr}{p_{\mathcal{W}}^{\mathcal{V}}} \\
        N(\mathcal{U}) \arrow{ur}{p_{\mathcal{U}}^{\mathcal{W}}} \arrow{rr}{p_{\mathcal{U}}^{\mathcal{V}}} && N(\mathcal{V}).
    \end{tikzcd}
    \]

    \item Let $P$ and $Q$ be simplicial complexes such that $P$ is a subcomplex of $N(\mathcal{U})$, $Q$ is a subcomplex of $N(\mathcal{V})$, and $p_{\mathcal{U}}^{\mathcal{V}}(P) \subseteq Q$ for any refinement map $p_{\mathcal{U}}^{\mathcal{V}}$. Assume further that for any two refinement maps $p_{\mathcal{U}}^{\mathcal{V}},q_{\mathcal{U}}^{\mathcal{V}}$ and for any simplex $ \sigma=[U_0,U_1,\dots,U_n]\in P$ the simplex spanned by $\{p_{\mathcal{U}}^{\mathcal{V}}(U_i)\}_{i=0}^{n}\cup\{q_{\mathcal{U}}^{\mathcal{V}}(U_i)\}_{i=0}^{n}$ lies in $Q$. Then the map
    \[
        (p_{\mathcal{U}}^{\mathcal{V}})^{*} \colon H^{i}(N(\mathcal{V}), Q) \longrightarrow H^{i}(N(\mathcal{U}), P) ,
    \]
    induced by any refinement map $p_{\mathcal{U}}^{\mathcal{V}}$, is trivial for all $i > k$.
\end{enumerate}
\end{lemma}
\begin{proof}
Let $g$ be a function satisfying the conclusion of Lemma~\ref{Lebesguenumber}. Then there exists a cover $\mathcal{W}$ of multiplicity at most $(k+1)$ such that
$\operatorname{diam}(\mathcal{W}) \le g(\lambda)
\quad\text{and}\quad
\mathcal{L}(\mathcal{W}) \ge \lambda .$
Consequently, for any cover $\mathcal{U}$ with $\operatorname{diam}(\mathcal{U}) \le \lambda$ we have $\mathcal{U} \preccurlyeq \mathcal{W}$.
Similarly, for any cover $\mathcal{V}$ with $\mathcal{L}(\mathcal{V}) \ge g(\lambda)$ we have $\mathcal{W} \preccurlyeq \mathcal{V}$.
This proves~(1).

\medskip

Assertions~(2) and~(3) are consequences of~(1).
Consider three covers $\mathcal{U}, \mathcal{W}, \mathcal{V}$ such that
$\mathcal{U} \preccurlyeq \mathcal{W} \preccurlyeq \mathcal{V}.$ Fix refinement maps
$p_{\mathcal{U}}^{\mathcal{V}},p_{\mathcal{U}}^{\mathcal{W}},p_{\mathcal{W}}^{\mathcal{V}}$ between them. Observe that both $p_{\mathcal{W}}^{\mathcal{V}} \circ p_{\mathcal{U}}^{\mathcal{W}}\quad\text{ and }p_{\mathcal{U}}^{\mathcal{V}}$
are refinement maps from $\mathcal{U}$ to $\mathcal{V}$; hence the induced maps are contiguous.

\medskip

Since $N(\mathcal{W})$ is at most $k$-dimensional, for any $i>k$ and any subcomplex $B \subseteq N(\mathcal{W})$ the cohomology group
$H^{i}(N(\mathcal{W}), B)$ is trivial.

Let $P \subseteq N(\mathcal{U})$ and $Q \subseteq N(\mathcal{V})$ be subcomplexes such that
$p_{\mathcal{U}}^{\mathcal{V}}(P) \subseteq Q.$
Let $B := (p_{\mathcal{W}}^{\mathcal{V}})^{-1}(Q)$.
We obtain the following commutative diagram of pairs:
\[
\begin{tikzcd}
    & (N(\mathcal{W}), B) \arrow{dr}{p_{\mathcal{W}}^{\mathcal{V}}} \\
    (N(\mathcal{U}), P) \arrow{ur}{p_{\mathcal{U}}^{\mathcal{W}}}
    \arrow{rr}{p_{\mathcal{W}}^{\mathcal{V}} \circ p_{\mathcal{U}}^{\mathcal{W}}}
    && (N(\mathcal{V}), Q).
\end{tikzcd}
\]

This induces the following commutative diagram in relative cohomology:
\[
\begin{tikzcd}
    & H^{i}(N(\mathcal{W}), B) \cong 0 \arrow{dr}{(p_{\mathcal{U}}^{\mathcal{W}})^{*}} \\
    H^{i}(N(\mathcal{V}), Q)
    \arrow{ur}{(p_{\mathcal{W}}^{\mathcal{V}})^{*}}
    \arrow{rr}{(p_{\mathcal{W}}^{\mathcal{V}} \circ p_{\mathcal{U}}^{\mathcal{W}})^{*}}
    && H^{i}(N(\mathcal{U}), P).
\end{tikzcd}
\]

Consequently,
$(p_{\mathcal{W}}^{\mathcal{V}} \circ p_{\mathcal{U}}^{\mathcal{W}})^{*} = 0.$
By Lemma~\ref{relativehomotopy},
$p_{\mathcal{W}}^{\mathcal{V}} \circ p_{\mathcal{U}}^{\mathcal{W}}$ is contiguous to $p_{\mathcal{U}}^{\mathcal{V}}$ relative to $P$.
Therefore,
\[
(p_{\mathcal{U}}^{\mathcal{V}})^{*}
=
(p_{\mathcal{W}}^{\mathcal{V}} \circ p_{\mathcal{U}}^{\mathcal{W}})^{*}
= 0,
\]
which proves~(3).
\end{proof}

\subsection{Alexander duality for manifolds with boundary}
We recall a version of the Alexander duality theorem that holds for manifolds with boundary.  For a triangulable space $M$ a subset $D\subseteq X$ is said to be a polyhedron if there exists a triangulation of $M$ with respect to which $D$ is a subcomplex. 
\begin{theorem}\cite[Theorem 72.3]{Munkres84}\label{poincarelefschetz}
Let $(M,\partial M)$ be a compact triangulable $d$-manifold with boundary. Then, there is a function $\phi$ that assigns to each $D$ polyhedron in $(M,\partial M)$ that contains $\partial M$, an isomorphism,
$$\phi_{D}:H^{k}(M,D)\rightarrow H_{d-k}(M\setminus D).$$
Moreover, this assignment is natural with respect of inclusions of polyhedra:
Let $E$ be a polyhedron containing $\partial M$ such that $D\subseteq E$ and let $i$ denote the inclusion map $i:D\hookrightarrow E$, then the following diagram commutes,
\begin{center}
\begin{tikzcd}
H^{k}(M,D)\arrow{r}{\phi_{D}}& H_{d-k}(M\setminus D)\\
H^{k}(M,E)\arrow{u}{i^\ast}\arrow{r}{\phi_{E}} & H_{d-k}(M\setminus E)\arrow{u}{i_\ast}\\
\end{tikzcd}
\end{center}

\end{theorem}

Below we discuss some of the details of the proof of the above Theorem \ref{poincarelefschetz}. These will be used in the proof of Lemma \ref{mainlemma}. For a complete proof we refer the reader to \cite{Munkres84}. 

\vspace{5mm}
\textit{Proof Sketch:}
Let $(M,\partial M)$ be a manifold with boundary and let $D$ be a polyhedron such that $\partial M \subseteq D$. We start by choosing an open neighborhood $U$ of $D$ such that $U$ deformation retracts onto $D$ and $M\setminus D$ deformation retracts onto $X \setminus U$. We denote $M\setminus U$ by $M^{U}$.  
 
$$
H_{d}(M,\partial M) \xrightarrow{m_\ast} H_{d}(M,\overline{U}) \xrightarrow{(k_\ast)^{-1}} H_{d}(M_U,\partial U).
$$

Here $m$ is the inclusion map. $k:(M^U,\partial U) \rightarrow(M,\overline{U})$ also denotes the inclusion map. $k_\ast$ is an isomorphism by Excision, and hence invertible.
Let $e_{U}$ denote the fundamental class of the manifold with boundary $(M^U,\partial U)$. $e_U \in H_n(M^U,\partial U)$ is the image of the fundamental class $e \in H_{n}(M,\partial M)$ under the homomorphism $ (k_\ast)^{-1}\circ m_\ast$. 
We are now ready to give a description of $\phi_D$. Consider the following sequence of maps:
$$
H^{k}(M,D)
\xleftarrow{i^{\ast}} H^{k}(M,\overline{U})
\xrightarrow{p^{\ast}} H^{k}(M^{U},\partial U)
\xrightarrow{-\cap e_{U}} H_{d-k}(M^U)
\xrightarrow{l_{\ast}} H_{d-k}(M\setminus D).
$$
The maps $i^\ast$, $p^{\ast}$, and $l_\ast$ in the above diagram are the maps induced by the canonical inclusions $i:D \hookrightarrow \overline{U}$, $p:\partial U \hookrightarrow \overline{U}, \textrm{ and } l:M^U \hookrightarrow M\setminus D$ . The maps $i^{\ast}$ and $l_{\ast}$ are isomorphisms since $i$ and $l$ are homotopy equivalences. The map $p^{\ast}$ is an isomorphism by Excision. Note that $(M_U,\partial U)$ is a manifold with boundary  and the map $(-\cap e_{U})$ denotes the cap product with the fundamental class $e_{U}\in H^{n}(M^{U},\partial U)$. This is also an isomorphism by Poincar\'e duality.
Thus, the isomorphism $\phi_D$ is given by
$\phi_D = l_\ast \circ (-\cap e_U) \circ p^{\ast} \circ (i^{\ast})^{-1}.$
The map $\phi_D$ does not depend on the choice of the open set $U$. 
\qedsymbol

\vspace{5mm}
We also recall a naturality property satisfied by cap products. This will be used in the proof of Lemma \ref{mainlemma}. 
\begin{theorem}\cite[pg. 241]{hatcher}\label{naturality}
    Let $f:X\rightarrow Y$ be a continuous map between CW complexes $X$ and $Y$. Given indices $k,l$, let $\cap_X :  H^{l}(X)\times H_k(X)\rightarrow H_{k-l}(X)$ and $\cap_Y : H_k(Y)\times H^{l}(Y) \rightarrow H_{k-l}(Y)$ denote the respective cap products. Then, for all $\alpha \in H_{k}(X), \varphi \in H^{l}(Y)$, 
    $$ \varphi \cap_{Y} f_{\ast}(\alpha) = f_{\ast} ( f^{\ast}(\varphi)\cap_X \alpha).$$
\end{theorem}

\subsection{Graphs of groups and spaces}
In this section we recall the definition of a graph of groups and the fundamental group associated to a graph of groups. In section \ref{applications}, we give an application towards understanding coarse embeddings of spaces with property $QF_d$ into fundamental groups of certain graphs of groups.  
\\
A graph of groups $(\Gamma,\mathcal{A})$ consists of the following data:
\begin{itemize}
    \item [(1)] An oriented and connected graph $\Gamma$. For each edge $e$, we denote by $\ori e$ its initial vertex and by $\ter e$ its terminal one.
    \item [(2)] An assignment of a \emph{vertex-group} $A_v$ to every vertex $v$ of $\Gamma$, an assignment of an \emph{edge-group} $A_e$ to every edge $e$, and an injective homomorphism $i_0 : A_e \to A_v$ (resp.\ $i_1 : A_e \to A_v$) whenever $\ori e = v$ (resp.\ $\ter e = v$).
\end{itemize}

Given a graph of groups $\pi_1(\Gamma,\mathcal{A})$ there is a canonical way of associating a group to it  called its \emph{fundamental group}, dented by $\pi_1(\Gamma,\mathcal{A})$. We briefly recall the topological definition of $\pi_1(\Gamma,\mathcal{A})$ was given in \cite{scott1979topological}. To a graph of groups $(\Gamma,\mathcal{A})$ one associates a \emph{graph of spaces} $X$. For each vertex $v$ of $\Gamma$ (resp.\ each edge $e$) take a finite simplicial complex $X_v$ (resp.\ $X_e$) such that $\pi_1(X_v) = A_v$ (resp.\ $\pi_1(X_e) = A_e$). Let $I$ be the unit interval. Then $X$ is defined by gluing the complexes $X_v$ and $X_e \times I$, for any $v$ vertex and $e$ edge, as follows. If $e$ is an edge and $i_0 : A_e \to A_{\ori e}$ is the associated injective homomorphism, let $f_0 : X_e \to X_{\ori e}$ be a simplicial map inducing it; and identify, for each $x \in X_e$,  $(x,0) \in X_e \times \left\{0\right\}$ to $f_0(x) \in X_{\ori e}$. Similarly, we identify $X_e \times \left\{1\right\} $ to $X_{\ter e}$. The fundamental group $\pi_1(\Gamma,\mathcal{A})$ is then defined as the fundamental group of $X$, $\pi_1(X) = \pi_1(\Gamma,\mathcal{A})$
The universal cover $\Tilde{X}$ is a union of copies of $\Tilde{X_v}$ and $\Tilde{X_e}\times I$. By giving each edge length one, $\Tilde{X}$ is quasi-isometric to $\pi_1(\Gamma,\mathcal{A})$. Moreover, the Bass-Serre tree $T$ of $\pi_1(\Gamma,\mathcal{A})$ can be obtained from $\Tilde{X}$, see \cite[Section 4]{scott1979topological}, by identifying each copy of $\Tilde{X_v}$ in $\Tilde{X}$ to a vertex and each copy of $\Tilde{X_e}\times I$ to a copy of $I$. This map $p : \Tilde{X}\to T$ is $\pi_1(\Gamma,\mathcal{A})$-equivariant. If $t$ is a vertex of $T$ (resp.\ a midpoint of an edge of $T$), we call $p^{-1}(t)$ a \emph{vertex-space} (resp.\ an \emph{edge-space}). It is easy to see that removing an edge-space separates $\Tilde{X}$ into two coarsely connected components with arbitrarily large balls. Therefore, the family of edge-spaces coarsely separates $\Tilde{X}$.
A useful structural result of Bensaid-Tessera-Genevois provides a dichotomy for coarse embeddings into the universal cover of a graph of spaces:
\begin{theorem}\label{thm:coarsesepembedding}\cite[Theorem 4.1]{Bensaid2024CoarseSA}
Let $X$ be a graph of spaces associated to a graph of groups $(\Gamma,\mathcal{A})$, $Z$ a coarsely connected metric space, and $f : Z \to \Tilde{X}$ a coarse embedding. Either $f(Z)$ is coarsely separated by the family of edge-spaces of $\Tilde{X}$ or it is contained in a bounded neighborhood of some vertex-space. 
\end{theorem}

\section{Reduction to the finitary version}\label{QF}
In this section, we reduce the proof of  Theorem \ref{main} to Theorem \ref{mainresult}. Given a space $X$ satisfying property $\textrm{QF}_d$ and a family $\mathcal{B}$ that coarsely separates $X$, we produce a family $\mathcal{A}$ such that $\mathcal{A}$ quasi-isometrically embeds in $\mathcal{B}$ and satisfies the hypotheses of Theorem \ref{mainresult}.
Theorem \ref{mainresult} can be thought of as a finitary version of Theorem \ref{main}. In Section \ref{section:mainresult} we give a proof of Theorem \ref{mainresult}. 

We first introduce property $QF_d$. $QF$ here stands for `quasi-flat'. \begin{definition}[Property $\textrm{QF}_d$]
    Let $d\in \mathbb{N}$. 
    A metric space $X$ is said to satisfy property $\textrm{QF}_d$ if there a exist constant $K\geq 1$ such that for all $x,y\in X$ there exists $L>0$ and a $K$-quasi-isometric embedding $q:\mathbb{R}^d\rightarrow X$ such that $x,y\in \operatorname{Im}(q)$. 
\end{definition}
Clearly, property $QF_d$ is a quasi-isometry invariant. 
\\
Examples : 
\begin{enumerate}
    \item $\mathbb{R}^{d}$ equipped with the $\ell_\infty$-metric trivially satisfies $\textrm{QF}_d$. More generally, any $d$-dimensional Euclidean building $X$ satisfies property $\textrm{QF}_d$ as any pair of points in $X$ lies in an isometrically embedded copy of $\mathbb{R}^d$. 
    \item Any (quasi-)geodesic metric space satisfies $\textrm{QF}_1$. More generally, a product $X=\Pi_{i=1}^{d} X_i$ of $d$ geodesic metric spaces $X_1,X_2,\dots, X_d$ satisfies property $\textrm{QF}_d$. 
\end{enumerate}

We first prove a lemma which shows that a path-connected metric space is $k$ coarsely connected for all $k>0$.

\begin{lemma}\label{points}
Let $X$ be a metric space that is path-connected in the metric topology. Let $x,y \in X$ and let $\gamma$ be a continuous path joining $x$ to $y$. Then, for any $k>0$, there exists a finite sequence of points
$x = t_0, t_1, \dots, t_r = y$
such that, for all $0 \le i \le r-1$, $d(t_i,t_{i+1}) \le k$ and $t_i \in \operatorname{Im}(\gamma)$.
\end{lemma}

\begin{proof}
Since $\operatorname{Im}(\gamma)$ is compact, there exists a finite collection $\mathcal{B}$ of open balls of radius $\frac{k}{2}$ that covers $\operatorname{Im}(\gamma)$. Let
$\mathcal{F} = \{ B \cap \operatorname{Im}(\gamma) \mid B \in \mathcal{B} \}.$
Let $B_0 \in \mathcal{F}$ be such that $x \in B_0$. Let $B_1$ be an element of $\mathcal{F} \setminus \{B_0\}$ such that $B_0 \cap B_1 \neq \varnothing$; such an element exists since $\operatorname{Im}(\gamma)$ is path-connected. Similarly, let $B_2$ be an element of $\mathcal{F} \setminus \{B_0,B_1\}$ such that $B_2 \cap B_1 \neq \varnothing$. Proceeding inductively, we obtain a sequence
$B_0, B_1, B_2, \dots, B_r$
of sets in $\mathcal{F}$ such that $B_i \cap B_{i+1} \neq \varnothing$ for each $0 \le i \le r-1$ and $x \in B_0$, $y \in B_r$.
For $1 \le i \le r-1$, let $t_i$ be a point in $B_i \cap B_{i+1}$, and set $t_0 = x$, $t_r = y$. Then the sequence
$x = t_0, t_1, \dots, t_r = y$
satisfies the required properties. Indeed, for all $0 \le i \le r-1$, since both $t_i$ and $t_{i+1}$ lie in the same ball of radius $\frac{k}{2}$, we have
$d(t_i,t_{i+1}) \le k.$
\end{proof}
We are now ready to state the main result of this section:
\begin{theorem}\label{reduction}
Let $X$ be a geodesic metric space that satisfies property $\textrm{QF}_d$.
Let $\mathcal{B}$ be a family of subsets that coarsely separates $X$. Then there exist a sequence of metric spaces $\mathcal{A}=\{A_n\}_n$ and two sequences $(L_n)_n$ and $(b_n)_n$ of positive real numbers tending to infinity such that the following conditions are satisfied:
\begin{enumerate}
    \item $A_n$ is a subset of $C_n=[0,L_n]^d$.
    \item $C_n \setminus A_n$ consists of more than one connected component, and there exist points $x_n,y_n\in C_n$ such that $x_n$ and $y_n$ belong to different connected components of $C_n \setminus A_n$, and
    $$d(x_n,A_n\cup \partial C_n)>b_n \quad\text{and}\quad d(y_n,A_n\cup \partial C_n)>b_n.$$
    \item $\mathcal{A}$ admits a quasi-isometric embedding into $\mathcal{B}$. In particular, $ \operatorname{sdim}(\mathcal{A})\leq \operatorname{asdim}(\mathcal{B})$.
\end{enumerate}
\end{theorem}
 
   \begin{proof} 
    Let $K\geq 1$ be such that any two points lie in the image of some $K$-quasi-isometric embedding $p$ of $\mathbb{R}^{d}$ into $X$. Let $K'=\max\{K,3K^2,3\}$. As noted in Remark \ref{remark:zero}, assume the constant $L$ in the definition of coarse separation to be zero. 
\par 
   First, we define the sequence $\{A_n\}_n$ as follows: Given $n\in \mathbb{N}$, there exists $B_n\in \mathcal{B}$ such that $X\setminus B_n$ contains more than one path component and there exist points $s_n,t_n\in X$ such that $s_n$ and $t_n$ lie in different path-components of $X\setminus B_n$ and $d(s_n,B_n), d(t_n,B_n)>\max\{\frac{n+K'}{K'},3K'\}$. Let $p_n:\mathbb{R}^{d}\rightarrow X$ be an $K$-quasi-isometric embedding such that $s_n,t_n\in Im(p_n)$. Let $q_n:Im(p_n)\rightarrow \mathbb{R}^d$ be a $K'$-quasi-inverse of $p_n$. Then, $q_n$ is a $K'$ quasi-isometry between $Im(p_n)$ and $\mathbb{R}^d$. Define $x_n:=q_n(s_n)$ and $y_n:=q_n(t_n)$. Let $C_n$ be a $d$-dimensional cube in $\mathbb{R}^{d}$ such that the interior of $C_n$ contains both $B(x_n,n)$ and $B(y_n,n)$. We assume that the sides of $C_n$ are parallel to the principle axes. Let $A_n=p_n^{-1}(N_{2K'}(B_n))\cap C_n$. Let $L_n$ denote the side length of $C_n$.

\par
   Next we show that $x_n$ and $y_n$ lie in distinct path-components of  $C_n\setminus A_n$. Let $\gamma$ be a continuous path in $C_n$ joining $x_n$ to $y_n$. We show that $\gamma$ necessarily intersects $A_n$. This is equivalent to showing that $p_n\circ \gamma$ intersects $N_{2K'}(B_n)$. Note that $p_n\circ \gamma$ joins $p_n\circ q_n(s_n)$ to $p_n\circ q_n(t_n)$. As $d(s_n,p_n\circ q_n(s_n)),d(t_n,p_n\circ q_n(t_n))\leq K'$ it follows that $ p_n\circ q_n(s_n)$ and $p_n\circ q_n(t_n)$ lie in distinct components of $X\setminus B_n$. By Lemma \ref{points}, there exists a sequence of points $x_n=u_1,u_2,\dots , u_r=y_n$ such that, for all $1\leq i\leq r-1$, $d(u_i,u_{i+1})\leq 1$ and $u_i\in \operatorname{Im}(\gamma)$. Consider the sequence $q_n(u_1),q_n(u_2),\dots,q_n(u_r)$. For all $ 1\leq i \leq r-1$, $d(q_n(u_i),q_n(u_{i+1}))\leq K'd(u_i,u_{i+1})+K'=2K'$. For $1\leq i\leq r-1$, let $\zeta_i$ be a geodesic that joins $q_n(u_i)$ to $q_n(u_{i+1})$. Let $\zeta$ denote the path joining $ p_n\circ q_n(s_n)$ to $p_n\circ q_n(t_n)$ by successively concatenating $\zeta_i$'s: $\zeta=\zeta_1\star\zeta_2\star \dots \star \zeta_r$. $\zeta$ necessarily intersects $B_n$ as $ p_n\circ q_n(s_n)$ and $p_n\circ q_n(t_n)$ lie in distinctcomponents of $ X\setminus B_n$. Suppose for some $1\leq k\leq r-1$, $ \zeta_k \cap B_n  \neq \emptyset$ then one of the endpoints of $\zeta_k$ must lie in $ N_{2K'}(B_n).$
   \par
    Now we show that $d(x_n,A_n),d(y_n,A_n)>n$. Let $z\in A_n$. Let $v\in N_{2K'}(B_n) $ be such that  $z=q_{n}(v)$. As $d(s_n,B_n)>\frac{n+K'}{K'}$, we have that $d(s_n,v)>\frac{n+K'}{K'}$. Using the fact that $q_n$ is a $K'$-quasi-isometry we get,  $d(z,x_n)=d(q_n(v),q_n(s_n))\geq K'd(s_n,v)-K'>n$. 
\par
   Now we verify (3). The map $p_n: A_n\rightarrow N_{2K'}(B_n)$ is a $K$-quasi-isometric embedding of $A_n$ into $N_{2K'}(B_n)$. For each $n$, one can construct a natural $4K'$-quasi-isometric embedding $r_n:N_{2K'}(B_n)\rightarrow B_n$ such that $r_n$ sends each $b\in N_{2K'}(B_n)$ to a point $r_n(b)$ such that each $d(b,r_n(b))\leq 2K'$.  Composing these two maps gives us the required quasi-isometric embedding. 
\end{proof}

\section{Connectivity and Regular Neighborhoods of Subcomplexes}\label{section:regnbd}

In this section we prove two lemmas about simplicial complexes. Both will be used in the proof of Theorem~\ref{mainresult}.

The first lemma gives a combinatorial criterion for when the complement of a subcomplex is path-connected.
\begin{lemma}\label{connected}
Let $K$ be a simplicial complex and let $L$ be a subcomplex of $K$. Let $Bd(K)$ denote the barycentric subdivision of $K$. Let $x,y \in K^{(0)} \setminus L^{(0)}$. Then there exists a continuous path $\gamma:[0,1]\rightarrow |K|\setminus |L|$ joining $x$ to $y$ if and only if there exists an edge-path in $Bd(K)$ joining $x$ to $y$ that does not intersect $Bd(L)$.
\end{lemma}

\begin{proof}
The “if” direction is immediate: any edge-path in $Bd(K)$ avoiding $Bd(L)$ is a continuous path in $|K|\setminus |L|$ joining $x$ to $y$.

For the converse, let $\gamma:[0,1]\to |K|\setminus |L|$ be a continuous path from $x$ to $y$. We modify $\gamma$ in two steps.

First, pass to a sufficiently fine barycentric subdivision. For $n\in \mathbb{N}$, let $Bd^n(K)$ denote the $n$-fold barycentric subdivision of $K$, and note that $Bd^n(L)$ is naturally a subcomplex of $Bd^n(K)$. Choose $n$ large enough so that the image of $\gamma$ is disjoint from the open star $\operatorname{st}(Bd^n(L))$. Let $U$ be the corresponding open subset of $|K|$. Then $\gamma$ is a map
$$
\gamma:[0,1]\to |K|\setminus U,
$$
and $|K|\setminus U$ is the geometric realization of $Bd^n(K)\setminus \operatorname{st}(Bd^n(L))$.

Applying the simplicial approximation theorem, there exists $m\in \mathbb{N}$ and a simplicial map
$$
g: Bd^m([0,1]) \to Bd^n(K)\setminus \operatorname{st}(Bd^n(L))
$$
which is homotopic to $\gamma$ relative to $\{0,1\}$. In particular, $g(0)=x$ and $g(1)=y$.

Let $t_i = i/m$ for $0\le i \le m$. Then each $g(t_i)$ is a vertex of $Bd^n(K)\setminus \operatorname{St}(Bd^n(L))$, hence lies in the interior of a unique simplex $\sigma_i$ of $K$. For each $0\le i \le m-1$, exactly one of the following holds:
\begin{enumerate}
\item $\sigma_i = \sigma_{i+1}$,
\item $\sigma_i$ is a face of $\sigma_{i+1}$,
\item $\sigma_{i+1}$ is a face of $\sigma_i$.
\end{enumerate}

Let $b_i$ denote the barycenter of $\sigma_i$. Since $x$ and $y$ are vertices of $K$, we have $b_0 = x$ and $b_m = y$. In each of the above cases, $[b_i,b_{i+1}]$ is an edge of $Bd(K)$. Thus the sequence $b_0,\dots,b_m$ determines an edge-path in $Bd(K)$ from $x$ to $y$.

Finally, since each $\sigma_i$ avoids $L$, none of the vertices $b_i$ lie in $Bd(L)$, and hence the resulting edge-path does not intersect $Bd(L)$.
\end{proof}
The second lemma proves the existence of an open neighborhood $U$ of a subcomplex $Y\subseteq X$ such that $U$ deformation retracts onto $Y$ and $X\setminus Y$ deformation retracts onto $X\setminus U$. 
\\ 

\begin{lemma}\label{nicenbd}
Let $X$ be a simplicial complex and $Y \subseteq X$ a subcomplex. Then there exists a neighborhood $U$ of $|Y|$ such that $U$ deformation retracts onto $|Y|$, and $|X|\setminus |Y|$ deformation retracts onto $|X|\setminus U$.    
\end{lemma}

\begin{proof}
Let $Bd(X)$ denote the barycentric subdivision of $X$, and let $Bd(Y)$ denote the barycentric subdivision of $Y$, viewed as a subcomplex of $Bd(X)$.

We first construct the neighborhood $U$. Any $n$-simplex $\sigma \in \operatorname{St}(Bd(Y))$ can be written as
$
[x_{0},x_{1},\dots,x_k,x_{k+1},\dots,x_{n}],
$
where $[x_{0},\dots,x_k]$ is a simplex in $Bd(Y)$ and $x_i \notin Bd(Y)^{(0)}$ for all $k+1 \leq i \leq n$. This uses the fact that $Bd(Y)$ is a flag complex: if $\sigma \in \operatorname{St}(Bd(Y))$ but $\sigma \notin Bd(Y)$, then at least one vertex lies outside $Bd(Y)$.

Any point $x \in \sigma$ can be written as
$
x = \sum_{i=0}^{n} t_i x_i, \qquad \sum_{i=0}^{n} t_i = 1.
$
For each such simplex $\sigma$, define
$$
U_{\sigma} = \left\{ x \in \sigma \,\middle|\, \sum_{i=k+1}^{n} t_i \in [0,\tfrac{1}{2}) \right\},
$$
and set $U = \bigcup_{\sigma \in \operatorname{St}(Bd(Y))} U_{\sigma}$.

We now define a deformation retraction $H$ of $U$ onto $|Y|$. For $x \in U_\sigma$, write $p = \sum_{i=0}^{k} t_i$. Define
$$
H(x,s) = (p + s(1-p))\left(\sum_{i=0}^{k} \frac{t_i}{p} x_i\right) + (1-s)\left(\sum_{i=k+1}^{n} t_i x_i\right).
$$
Then $H(\cdot,s)$ fixes $[x_0,\dots,x_k]$ pointwise (i.e. when $p=1$), and $H(x,1) \in Bd(Y)$ for all $x \in U$. Moreover, the definitions agree on intersections of simplices, since they are compatible under restriction to faces. Hence $H$ defines a deformation retraction of $U$ onto $|Y|$.

Next, we construct a deformation retraction of $|X|\setminus |Y|$ onto $|X|\setminus U$. Let $\sigma = [x_0,\dots,x_n] \in \operatorname{St}(Bd(Y))$ as above. Then
$$
(|X|\setminus |Y|)\cap \sigma = \left\{ x \in \sigma \,\middle|\, \sum_{i=k+1}^{n} t_i \in (0,1] \right\},
$$
while
$$
(|X|\setminus U)\cap \sigma = \left\{ x \in \sigma \,\middle|\, \sum_{i=k+1}^{n} t_i \in [\tfrac{1}{2},1] \right\}.
$$

Write $p = \sum_{i=0}^{k} t_i$ and $q = \sum_{i=k+1}^{n} t_i$. Define
$$
F(x,s) =
\begin{cases}
x & \text{if } q \in [\tfrac{1}{2},1], \\
(p + s(\tfrac{1}{2}-p))\left(\sum_{i=0}^{k} \frac{t_i}{p} x_i\right)
+ (q + s(\tfrac{1}{2}-q))\left(\sum_{i=k+1}^{n} \frac{t_i}{q} x_i\right)
& \text{if } q \in (0,\tfrac{1}{2}).
\end{cases}
$$

 $F(\cdot,s)$ restricts to the identity on $(|X|\setminus U)\cap \sigma$ for all $s$, and $F(x,1) \in (|X|\setminus U)\cap \sigma$ for all $x \in (|X|\setminus |Y|)\cap \sigma$. As before, these maps are compatible on overlaps of simplices, and hence glue to a global deformation retraction of $|X|\setminus |Y|$ onto $|X|\setminus U$.
\end{proof}
\begin{corollary}\label{cor:nbd}
Let $f:X\rightarrow Y$ be a simplicial map between simplicial complexes $X$ and $Y$. Let $L$ be a subcomplex of $Y$ and let $K = f^{-1}(L)$. Then there exists an open neighborhood $V$ of $L$ such that 
\begin{enumerate}
\item $V$ deformation retracts onto $L$, 
\item $Y \setminus L$ deformation retracts onto $Y \setminus V$,
\item $U = f^{-1}(V)$ deformation retracts onto $K$, and 
\item $X \setminus K$ deformation retracts onto $X \setminus U$.
\end{enumerate}
\end{corollary}

\begin{proof}
Let $Bd(X)$ and $Bd(Y)$ denote the barycentric subdivisions of $X$ and $Y$, respectively. Similarly, let $Bd(K)$ and $Bd(L)$ denote the barycentric subdivisions of $K$ and $L$. By a slight abuse of notation, we denote the induced simplicial map $Bd(X)\rightarrow Bd(Y)$ by $f$. Then $f^{-1}(Bd(L)) = Bd(K)$.

As in the proof of Lemma~3.2, any $n$-simplex $\sigma \in \operatorname{St}(Bd(K))$ can be written as\\ $
[x_{0},x_{1},\dots,x_k,x_{k+1},\dots,x_{n}],$
where $[x_{0},\dots,x_k]$ is a simplex in $Bd(K)$ and $x_i \notin Bd(K)^{(0)}$ for all $k+1 \leq i \leq n$. 

For each such simplex $\sigma$, define
$$
U_{\sigma} = \left\{ x = \sum_{i=0}^{n} t_i x_i \,\middle|\, \sum_{i=k+1}^{n} t_i \in [0,\tfrac{1}{2}) \right\},
$$
and set $U = \bigcup_{\sigma \in \operatorname{St}(Bd(K))} U_{\sigma}$.

By Lemma~3.2, $U$ deformation retracts onto $|K|$, and $|X|\setminus U$ deformation retracts onto $|X|\setminus K$.

We define $V \subseteq Bd(Y)$ analogously using $\operatorname{St}(Bd(L))$. For each simplex $\sigma \in \operatorname{St}(Bd(L))$, define $V_\sigma$ by the same condition, and set $V = \bigcup V_\sigma$. Again by Lemma~3.2, $V$ deformation retracts onto $|L|$, and $|Y|\setminus V$ deformation retracts onto $|Y|\setminus L$.

Finally, by construction of the neighborhoods and the compatibility of $f$ with barycentric subdivision, we have $f^{-1}(V) = U$.
\end{proof}

\section{Covers of cubes with  nerves homeomorphic to $\mathbb{D}^d$}\label{section:covers}
    In this section we describe a construction which, given $\epsilon \in (0,1)$ and $L>0$, produces a cover $\mathcal{U}_\epsilon$ of the cube $C_L = [0,L]^d$ with the following properties:
\begin{itemize}
\item $\operatorname{mesh}(\mathcal{U}_\epsilon) < \epsilon L$,
\item $\mathcal{L}(\mathcal{U}_\epsilon) \ge\epsilon LM$, where $M$ depends only on $d$,
\item $|N(\mathcal{U}_\epsilon)|$ is homeomorphic to $\mathbb{D}^d$.
\end{itemize}
We first outline this construction for $L=1$ and then generalize it to arbitrary $L$  by simply dilating by $L$. 
\\
We begin with a lemma providing a triangulation of a cube of side length $\epsilon$ using only its vertices and with controlled simplex diameter.

\begin{lemma}\label{simplex}
Let $\epsilon > 0$, and let $C_\epsilon = [0,\epsilon]^d \subseteq \mathbb{R}^d$ be equipped with the $\ell_\infty$-metric. Then $C_\epsilon$ admits a triangulation whose vertex set consists exactly of the vertices of $C_\epsilon$. Consequently, every simplex in the triangulation has diameter $\epsilon$. Moreover, the induced triangulations of opposite codimension one faces agree under the natural translation.
\end{lemma}

\begin{proof}
We argue by induction on $d$. For $d=2$, the square $C_\epsilon$ can be decomposed into two triangles by adding a diagonal; each simplex has diameter $\epsilon$ in the $\ell_\infty$-metric.

Assume the statement holds in dimension $d-1$. Let $C = [0,\epsilon]^{d-1}$ be triangulated accordingly. Then $C_\epsilon = C \times [0,\epsilon]$. For each $(d-1)$-simplex $\sigma \subseteq C$, the prism $\sigma \times [0,\epsilon]$  can then be decomposed into $d$-simplices as follows. We first order the vertices of $C$ lexicographically. Let $\sigma=[v_0,v_1,\dots,v_{d-1}]$ be a simplex of $C$, with the vertices indexed so that $v_0<v_1<\cdots<v_{d-1}$. Then, $ \sigma \times \{0\}=[v_0^{0},v_1^{0},\dots ,v_{d-1}^{0}]$, where $v_{i}^{0}=v_{i}\times \{0\}, 1\leq i\leq d-1$. Also,  $ \sigma \times \{\epsilon\}=[v_0^{1},v_1^{1},\dots ,v_{d-1}^{1}]$ where $v_{i}^{1}=v_{i}\times \{1\}, 1\leq i\leq d-1$. 
    Then, 
    $ \sigma \times [0,\epsilon]$ is the union of simplices of the form $[v_0^{0},\dots,v_{k-1}^{0},v_k^{0},v_{k}^{1},v_{k+1}^{1},\dots,v_{d-1}^{1}]$ where $k$ ranges from $0$ to $d-1$ (cf. Figure \ref{fig:2}). These decompositions agree on adjacent prisms and give rise to a global triangulation of $C_{\epsilon}$. 
\end{proof}

\begin{figure}[h!]
    \centering
  \includegraphics[width=250pt]{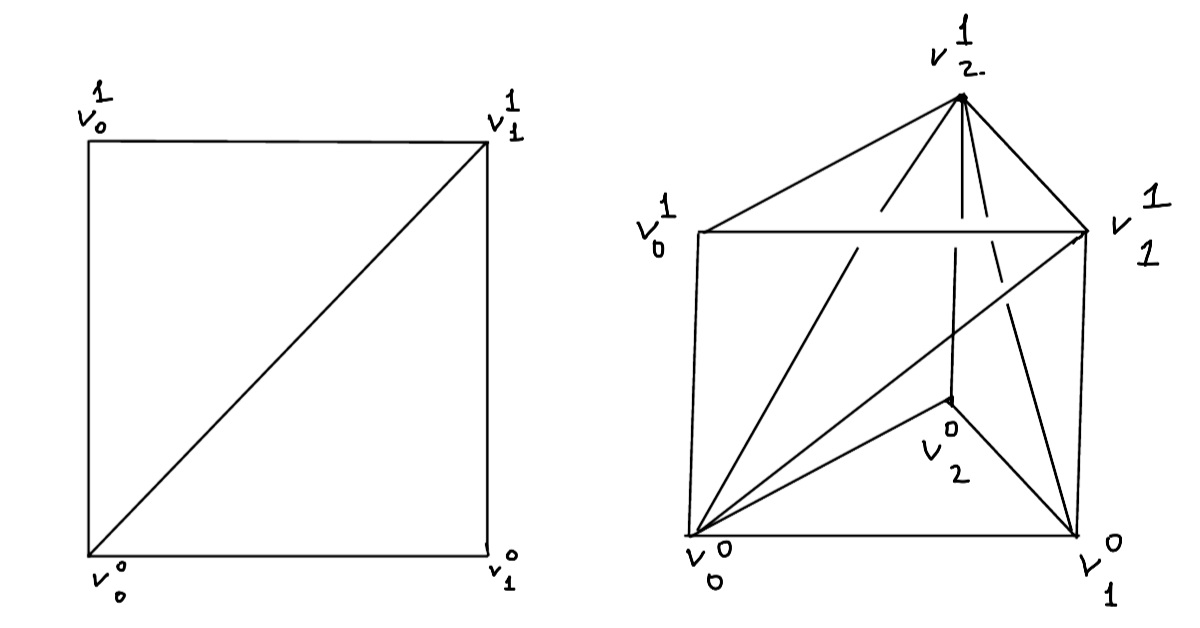}
            \caption{}
            \label{fig:2}
\end{figure}

We now define a constant $A = A(d) \in (0,1)$. Let $T$ be a triangulation of the unit cube $C_1 = [0,1]^d$ as above, so that every edge of $T$ is either an edge of the cube or a diagonal of a face. Let $Bd(T)$ denote its barycentric subdivision, and define
\begin{equation}
  B(T) = \min \{ d(\sigma,\sigma') \mid \sigma, \sigma' \in Bd(T),\ \sigma \cap \sigma' = \emptyset \}  
\end{equation}

 Set
\begin{equation}\label{constant}
   A(d) = \min_T B(T), 
\end{equation}

where the minimum is taken over all such triangulations $T$.
\medskip
\\
We are now ready to describe the construction for $L=1$.

\begin{center}
\textbf{The construction}
\end{center}

\begin{enumerate}
\item[\textbf{Step 1:}]
Given $\epsilon \in (0,1)$,
choose $n \in \mathbb{N}$ such that
$
\frac{\epsilon}{2} \le \frac{3}{n} < \epsilon.
$
Partition $[0,1]$ into $n$ equal intervals,
$
\mathcal{P} = \left\{ \left[\frac{k}{n}, \frac{k+1}{n}\right] \right\}_{k=0}^{n-1}.
$
Taking $d$-fold products yields a partition $\mathcal{Q}$ of $C_1$ into cubes:
$
\mathcal{Q} = \{ A_1 \times \cdots \times A_d \mid A_i \in \mathcal{P} \}.
$
Each element of $\mathcal{Q}$ has diameter $1/n$.
Next we triangulate each cube in $\mathcal{Q}$ as follows. 
Using Lemma~\ref{simplex}, triangulate the cube $[0,\frac{1}{n}]^d$. Obtain a triangulation of each of the cubes in $\mathcal{Q}$ by translating the triangulation of $[0,\frac{1}{n}]^d$ by an appropriate vector. This gives rise to a  triangulation of each cube in $\mathcal{Q}$ such that so that adjacent cubes induce compatible triangulations on their common faces. This yields a  triangulation $T$ of $C_1$ in which every simplex of positive dimension has diameter $1/n$. Moreover, $T$ is a flag complex. Figure \ref{fig3}
shows a triangulation of the square obtained for $n=3$. 

\begin{figure}[H]
    \centering
  \includegraphics[width=100pt]{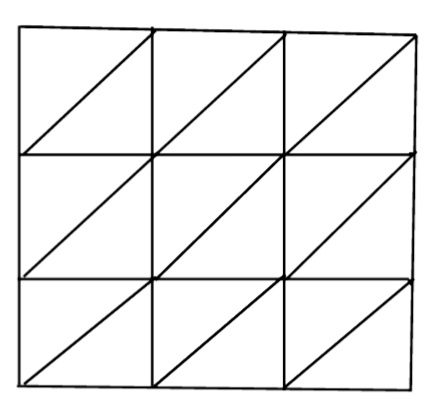}
  \caption{}
  \label{fig3}
\end{figure}

\item[\textbf{Step 2:}]
Let $Bd(T)$ denote the barycentric subdivision of $T$. For each vertex $p \in T^{(0)}$, define
$$
V_p = St_{Bd(T)}(p),
$$
and set
$$
\mathcal{V} = \{ V_p \}_{p \in T^{(0)}}.
$$
If $x \in V_p$, then $x$ lies in a simplex $\sigma$ of $T$ having $p$ as a vertex. Hence $d(x,p) \le \frac{1}{n}$, and therefore
$\operatorname{diam}(V_p) \le \frac{2}{n}.$ In Figure \ref{fig4} we show what $V_p$ looks like for a typical point $p$ in $T$ for the triangulation of the square obtained after taking $n=3$. 
\end{enumerate}
\begin{figure}[H]
    \centering
  \includegraphics[width=100pt]{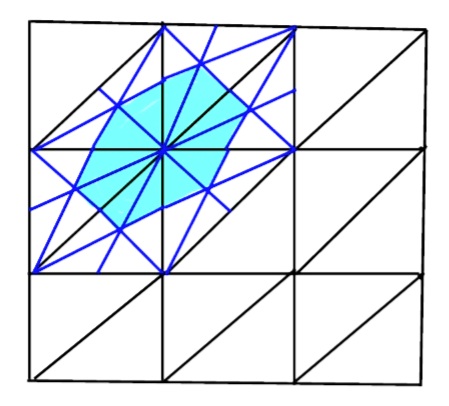}
          \caption{}
          \label{fig4}
\end{figure}
    Before giving the final step in the construction we make an important observation about the sets in $\mathcal{V}$. 

  \begin{lemma} \label{vpdist}
  There exists a constant $B=B(d)\in (0,1)$ such that for $p,q\in T^{(0)}$, either $V_p\cap V_{q}\neq \varnothing$ or $d(V_{p},V_{q})\geq\frac{B}{n}$.
\end{lemma}
    \begin{proof}
    Suppose $p,q$ are adjacent in $T^{(1)}$ then $V_{p}$ intersects $V_{q}$. Now suppose $p$ and $q$ are not adjacent in $T^{(1)}$. We divide the argument into two cases: 1) $d(p,q)= \frac{1}{n}$, and  2) $d(p,q)\geq \frac{2}{n}$. 
    The first case corresponds to points $p,q$ that are vertices of a cube $C_{p,q}\in \mathcal{Q}$ but are not adjacent in $T^{(1)}$. We deal separately with $ V_{q}\cap C_{p,q}$ and $V_{p}\setminus C_{p,q}$. Note that $d(V_{p},V_{q}\cap C_{p,q})\geq \frac{A}{n}$, where $A$ is the constant as defined in \ref{constant}. Also note that $d(p,V_{q}\setminus  C_{p,q})\geq \frac{1}{n}$. For each $p\in T^{(0)}$, $V_{p}\subseteq B(p, \frac{d}{(d+1)n})$ (Lemma \ref{lemma:bddiameter}). Therefore $d(V_{p},V_{q}\setminus  C_{p,q})\geq\frac{1}{n}-\frac{d}{(d+1)n}=\frac{1}{(d+1)n}$
    The second case is easier to deal with. For each $p\in T^{(0)}$, $V_{p}\subseteq B(p, \frac{d}{(d+1)n})$. If $x\in V_{p},y\in V_{q}$, then $d(x,y)\geq d(p,q)-d(p,x)-d(q,y)\geq \frac{2}{n}-\frac{2d}{(d+1)n}=\frac{2}{(d+1)n}$. 
    To summarize, if $B=\textrm{min}\{A(d),\frac{1}{(d+1)}\}$ then for $p,q\in T^{(0)}$, either $V_p\cap V_{q}\neq \varnothing$ or $d(V_{p},V_{q})\geq\frac{B}{n}$.
    \end{proof}
    
    \begin{enumerate}
        \item [\textbf{Step 3:}] For each $p\in T^{(0)}$ define $U_p =N_{\frac{2B}{5n}}(V_{p})$. Define $\mathcal{U}_{\epsilon}$ to be the cover $ \mathcal{U}_{\epsilon}=\{U_p| p\in T^{(0)}\}$.
    \end{enumerate}
   This completes the construction. Now we collect some properties of the the covers $\mathcal{U}_\epsilon$.

     \begin{lemma}\label{lemone}
    Let $M=\frac{B}{3}$. Then, $\operatorname{mesh}(\mathcal{U}_{\epsilon})<\epsilon $
        and $\mathcal{L}(\mathcal{U}_\epsilon)\geq{\epsilon}{M}.$        
    \end{lemma}

    \begin{proof}

         For any $U\in \mathcal{U_{\epsilon}}$, $\operatorname{mesh}(U)\leq \frac{2}{n}+\frac{2B}{5n}\leq \frac{3}{n}<\epsilon$. 
        Let $x\in V_p$ for some $p\in T^{(0)}$. Then, $B(x,\frac{2B}{6n})\subseteq N_{\frac{2B}{5n}}(V_p)$. This shows that $\mathcal{L}(\mathcal{U}_\epsilon)\geq {\epsilon}{M}$.
    \end{proof}

    \begin{lemma}\label{lemtwo}
The geometric realization $|N(\mathcal{U}_{\epsilon})|$ of $N(\mathcal{U}_{\epsilon})$ is homeomorphic to $\mathbb{D}^{d}$.         
    \end{lemma}

\begin{proof}
 It suffices to show that $N(\mathcal{U}_\epsilon)$ is simplicially isomorphic to $T$ as $ T$ is a triangulation of $C_1$ which is itself homeomorphic to $\mathbb{D}^d$.  We first show that $N(\mathcal{V})$ is isomorphic to $T$. 
 
\vspace{5mm}
\textit{Claim 1}:
    $N(\mathcal{V})$ is isomorphic to $T$.
 
\vspace{5mm}
\textit{Proof of Claim 1:}
    We show that the map that sends $p$ to $V_p$ defines a simplicial isomorphism between the two complexes. Suppose $[p_0,p_1,\dots , p_k]\in T$ then $V_{p_1}\cap V_{p_2}\cap \dots \cap V_{p_k}\neq \emptyset$. Consequently, $[V_{p_0},V_{p_1},\dots ,V_{p_k}]\in N(\mathcal{V})$. 
    Now let $V_{p_{0}},V_{p_1},\dots , V_{p_k}$ be elements of $\mathcal{V}$ such that the $V_{p_0}\cap V_{p_1}\cap \dots \cap V_{p_k}\neq \emptyset$. As $V_{p_i}$'s are subcomplexes of $Bd(T)$ their intersection is also a subcomplex. Let $b$ be a vertex of $Bd(T)$ such that $b\in V_{p_1}\cap V_{p_2}\cap \dots \cap V_{p_k}$. By definition, $b$ is a barycenter of some simplex in $T$ i.e., $b=\frac{1}{l+1}(v_0+v_1+\dots +v_l) $ for some simplex $[v_0,v_1,\dots , v_l]\in T$. For each $i$, since $b\in \operatorname{St}(p_i)$ it must be that $p_i$ is equal to one of the $v_j$'s. It follows that $p_i$'s span a simplex in $T$, indeed it is a face of the simplex $[v_0,v_1,\dots , v_l]$.  
\qedsymbol
 
\vspace{5mm}
\textit{Claim 2:}
    $N(\mathcal{U}_\epsilon)$ is isomorphic to $N(\mathcal{V})$.
 
\vspace{5mm}
\textit{Proof of Claim 2:}
We show that the map that sends $V_p$ to $U_p$ is a an isomorphism. Suppose $ N_\frac{2B}{5n}(V_{p_0})\cap N_\frac{2B}{5n}(V_{p_1})\cap\dots \cap N_\frac{2B}{5n}(V_{p_k}) \neq \emptyset$. Then, $N_\frac{2B}{5n}(V_{p_i})\cap N_\frac{2B}{5n}(V_{p_j})\neq \emptyset$ for all $ 0\leq i \leq j \leq k$. This implies that $d(V_{p_i},V_{p_j})\leq \frac{4B}{5n}$. It follows from Lemma \ref{vpdist} that for each $ 0\leq i \leq j \leq k$,  $V_{p_i}\cap V_{p_j}\neq \emptyset$. As a result $(V_{p_i},V_{p_j})$ is an edge in $N(\mathcal{V})$ for all $ 0\leq i< j \leq k$. As $N(\mathcal{V})$ is a flag complex (it is isomorphic to $T$) the $V_{p_i}$'s span a simplex in $N(\mathcal{V})$. 
\qedsymbol
    \end{proof}

\begin{remark}\label{remark:point}
For each simplex $\sigma\in [U_{p_{0}},\dots ,U_{p_{t}}]$ let $z_{\sigma}$ denote the barycenter of the simplex $ [p_0,\dots,p_t]\in T$. Then the point $z_\sigma\in \cap_{j=0}^{t}U_{p_{j}}$. Furthermore for each codimension one face $\sigma'=[U_{p_{0}},\dots \hat{U}_{p_{k}}\dots,U_{p_{t}}]$ of $\sigma$ the straight line joining $z_{\sigma'}$ and $z_\sigma$ lies entirely in the set $ U_{p_{0}}\cap\dots \cap\hat{U}_{p_{k}}\cap\dots\cap U_{p_{t}}$. 
\end{remark}

\begin{lemma}
    Let $ \epsilon_1,\epsilon_2\in (0,1)$ be such that $\epsilon_1 < {\epsilon_2}{M}$, so that $\mathcal{U}_{\epsilon_{1}}\preccurlyeq \mathcal{\mathcal{U}}_{\epsilon_{2}}$. Then, for any refinement map $f:N(\mathcal{U}_{\epsilon_{1}})\rightarrow N(\mathcal{U}_{\epsilon_{2}})$, $f(\partial N(\mathcal{U}_{\epsilon_{1}}))\subseteq \partial N(\mathcal{U}_{\epsilon_{2}})$ for all $d$.
\end{lemma}

\begin{proof}
     Let $n\in \mathbb{N}$ such that $\frac{\epsilon_2}{2}\leq \frac{3}{n} < \epsilon_2$. Let $T$ denote the triangulation of $C_1$ obtained while performing step 1. of the construction given above for $\epsilon=\epsilon_1$. Similarly, let $S$ denote the triangulation of $C_1$ for $\epsilon=\epsilon_2$. 
     Suppose $U_p\in \partial N(\mathcal{U}_{\epsilon_{1}})^{(0)}$ and let $f(U_p)= V\in \mathcal{U}_{\epsilon_{2}}$. This implies that $ U_p \subseteq V$. Using the fact that $T$ is isomorphic to $N(\mathcal{U}_{\epsilon_{1}})$ via the map $p\rightarrow U_p$ we get that  $ p\in \partial C_1$. Let $q\in S$ such that $V= N_{\frac{2B}{5n}}(St_{Bd(T)}(q))$.  Assume for the sake of contradiction that $V$ is a vertex in the interior of  $N(\mathcal{U}_{\epsilon_{2}})$.  This is equivalent to assuming that $q$ belongs to the interior of $C_1$. It follows that $d(q, \partial C_1)\geq\frac{1}{n}$. 
     Recall that the diameter of any simplex in the barycentric subdivision $Bd(S)$ is at most $ \frac{d}{(d+1)n}$. As a result, for any $ x\in V=N_{\frac{2B}{5n}}(\operatorname{St}(q))$, $d(q,x)< \frac{d}{(d+1)n} + \frac{2}{5(d+1)n}$. 
     For any $ x\in V$, 
     $$ d(x,\partial C_1)\geq d(q, \partial C_1)-d(q,x)\geq \frac{1}{n} - \frac{d}{(d+1)n} - \frac{2}{5(d+1)n}>0.$$
     $d(V,\partial C_1)>0$. This contradicts the fact that $p\in V \cap \partial C_1$. 
     \par
     Now let $[U_0,U_1,\dots , U_k]$ be a $k$-simplex in $ N(\mathcal{U}_{\epsilon_1})$. It follows via the isomorphism stated in the Claim 1. of the proof of Lemma \ref{lemtwo} is a simplex in the boundary $\partial N(\mathcal{U}_{\epsilon_1})$ if and only if the set $ \cap_{i=0}^{k} U_i$ intersects $\partial C_1$ nontrivially. For $1\leq i\leq k$, let $V_i=p(U_i)$. As $ U_i\subset V_i$ it follows that $\cap_{i=0}^{n} V_i$ intersects $C_1$ nontrivially. As a result, $[V_0,V_1,\dots , V_k]$ is contained in the boundary $\partial N(\mathcal{U}_{\epsilon_2})$. 
 \end{proof}
We prove another useful lemma which characterizes simplices that are contained in the boundary $\partial(N(\mathcal{U}_{\epsilon}))$.  

\begin{lemma}\label{lem:boundary}
  Let $[U_{0},U_{1},\dots,U_{t}]$ be a simplex of $N(\mathcal{U}_{\epsilon})$ such that $d(\cap_{i=0}^{k}U_i,\partial C_1)< \frac{\epsilon}{(d+1)10}$. Then, $[U_{0},U_{1},\dots,U_{t}]$ lies in $\partial N(\mathcal{U}_{\epsilon})$.
\end{lemma}

\begin{proof}
We prove the contrapositive of the above statement. That is, if $[U_0,U_1,\dots ,U_t]$ is not a simplex of $\partial N(\mathcal{U}_{\epsilon})$ then 
$d(\cap_{i=0}^{k}U_i,\partial C_1)\geq \frac{\epsilon}{(d+1)10}$.

For $1\leq i\leq t$ let $p_i\in T^{(0)}$ be such that $U_{i}=U_{p_i}$. Since the simplex $[p_0,\dots ,p_t]$ does not belong to the boundary $\partial C_1$ at least one of the vertices say $p_j$ lies outside $\partial C_1$. Consequently, $d(p_j,\partial C_1)>\frac{1}{n}$. 
    Recall that the diameter of any simplex in the barycentric subdivision $Bd(T)$ is at most $ \frac{d}{(d+1)n}$. As a result, for any $ x\in U_{p_i}=N_{\frac{2B}{5n}}(\operatorname{St}(q))$, $d(q,x)\leq \frac{d}{(d+1)n} + \frac{2}{5(d+1)n}$. 
     For any $ x\in U_{p_j}$, 
     $$ d(x,\partial C_1)\geq d(q, \partial C_1)-d(q,x)\geq \frac{1}{n} - \frac{d}{(d+1)n} - \frac{2}{5(d+1)n}=\frac{3}{5(d+1)n}.$$ 
Consequently, $$d(\cap_{i=0}^{k}U_{p_i},\partial C_1)\geq d(U_{p_j},\partial C_1)\geq \frac{3}{5(d+1)n}\geq \frac{\epsilon}{(d+1)10}.$$
\end{proof}

\begin{lemma}\label{lemthree}
  Let $ \epsilon_1,\epsilon_2 \in (0,1)$ be such that $\epsilon_1 < \epsilon_2 M$. Then, for any refinement map $f$,  $$f_{\ast}:\widetilde{H}_{d}(N(\mathcal{U}_{\epsilon_{1}}),\partial N(\mathcal{U}_{\epsilon_{1}}))\rightarrow \widetilde{H}_d( N(\mathcal{U}_{\epsilon_{2}}),\partial N(\mathcal{U}_{\epsilon_{2}}))$$ is an isomorphism.
\end{lemma}
\begin{proof}
   Let $\epsilon_1<\frac{\epsilon_2}{M}$.
   \par
   For $i=1,2$, the long exact of the pair $(N(\mathcal{U}_{\epsilon_{i}}),\partial N(\mathcal{U}_{\epsilon_{i}}))$ yields isomorphisms
   \begin{equation*}
      \delta_{i}:\widetilde{H}_d(N(\mathcal{U}_{\epsilon_{i}}),\partial N(\mathcal{U}_{\epsilon_{i}}))\rightarrow \widetilde{H}_{d-1}(\partial N(\mathcal{U}_{\epsilon_{i}})).
   \end{equation*}

Furthermore, by naturality of the long exact sequence the following diagram commutes:
\begin{center}
    \begin{tikzcd}
    \widetilde{H}_d(N(\mathcal{U}_{\epsilon_{1}}),\partial N(\mathcal{U}_{\epsilon_{1}}))\arrow{r}{f_\ast}\arrow{d}{\delta_1}&              \widetilde{H}_d(N(\mathcal{U}_{\epsilon_{2}}),\partial N(\mathcal{U}_{\epsilon_{2}}))\arrow{d}{\delta_2}\\
    \widetilde{H}_{d-1}(\partial N(\mathcal{U}_{\epsilon_{1}}))\arrow{r}{f_\ast}& \widetilde{H}_{d-1}(\partial N(\mathcal{U}_{\epsilon_{2}}))\\
    \end{tikzcd}
    \end{center}
    Furthermore, the maps $\delta_1$ and $\delta_2$ are isomorphisms.
    Consequently, the map \\$f_\ast: \widetilde{H}_d(N(\mathcal{U}_{\epsilon_{1}}),\partial N(\mathcal{U}_{\epsilon_{1}}))\rightarrow \widetilde{H}_d(N(\mathcal{U}_{\epsilon_{2}}),\partial N(\mathcal{U}_{\epsilon_{2}}))$ is an isomorphism if and only if the map $f_\ast: \widetilde{H}_{d-1}(\partial N(\mathcal{U}_{\epsilon_{1}}))\rightarrow \widetilde{H}_{d-1}(\partial  N(\mathcal{U}_{\epsilon_{2}}))$ is an isomorphism.
    \\
    We use induction on $d$ with base case $d=1$.

\vspace{5mm}
\textit{Base Case}:
In this case $C_1$ is just the unit interval $[0,1]$. Let $m\in \mathbb{N}$ be such that $\frac{\epsilon_1}{2} \le \frac{3}{m} < \epsilon_1.$ Performing step 1. of the above given construction yields the partition of $[0,1]$ into $m$ segments of equal length $ \{[\frac{i}{m},\frac{i+1}{m}]\}_{i=0}^{m-1}$. The cover $\mathcal{U}_{\epsilon_{1}}$ is given by thickening each of these intervals by $\frac{2B}{5n}$ at both ends: $\{[0,\frac{1}{m}+\frac{2B}{5m}),(\frac{1}{m}-\frac{2B}{5m},\frac{2}{m}+\frac{2B}{5m}),\dots,(\frac{m-1}{m}-\frac{2B}{5m},1]\}$.
\par
Similarly, let $n\in \mathbb{N}$ be such that $\frac{\epsilon_2}{2} \le \frac{3}{n} < \epsilon_2.$ The cover $\mathcal{U}_2$ is $\{[0,\frac{1}{n}+\frac{2B}{5n}),(\frac{1}{n}-\frac{2B}{5n},\frac{2}{n}+\frac{2B}{5n}),(\frac{2}{n}-\frac{2B}{5n},\frac{3}{n}+\frac{2B}{5n}),\dots,(\frac{n-1}{n}-\frac{2B}{5n},1]\}$.
\par
For $i=1,2$, the nerve $N(\mathcal{U}_{\epsilon_i})$ is homeomorphic to  the unit interval $[0,1]$. The boundary $\partial N(\mathcal{U}_{\epsilon_1})$ consists of two (non-adjacent) vertices corresponding to the sets $\{[0,\frac{1}{m}+\frac{2B}{5m}),(\frac{m-1}{m}-\frac{2B}{5m},1]\}$. Analogously the boundary
$\partial N(\mathcal{U}_{\epsilon_2})$ consists of two isolated vertices $\{[0,\frac{1}{n}+\frac{2B}{5n}),(\frac{n-1}{n}-\frac{2B}{5n},1]\}$.
Under $f$, $[0,\frac{1}{m}+\frac{2B}{5m})$ maps to the set $[0,\frac{1}{n}+\frac{2B}{5n})$ as this is the only set in $ \mathcal{U}_{\epsilon_2}$ that contains $0$. Similarly, $(\frac{m-1}{m}-\frac{2B}{5m},1]$ must necessarily map to $(\frac{n-1}{n}-\frac{2B}{5n},1]$. Thus $f$ is a homeomorphism between $\partial N(\mathcal{U}_{\epsilon_{1}})$ and $\partial  N(\mathcal{U}_{\epsilon_{2}})$. As a result,
the map $f_\ast: \widetilde{H}_{0}(\partial N(\mathcal{U}_{\epsilon_{1}}))\rightarrow \widetilde{H}_{0}(\partial  N(\mathcal{U}_{\epsilon_{2}}))$ is an isomorphism. In light of above discussion, it follows that $f_\ast: \widetilde{H}_1(N(\mathcal{U}_{\epsilon_{1}}),\partial N(\mathcal{U}_{\epsilon_{1}}))\rightarrow \widetilde{H}_1(N(\mathcal{U}_{\epsilon_{2}}),\partial N(\mathcal{U}_{\epsilon_{2}}))$ is an isomorphism. This completes the proof for $d=1$.
\vspace{5mm}
\\
\textit{Induction Step}:
  Let $D$ be a fixed $(d-1)$ dimensional face of $C_1$. For $i=1,2$, let $T_i$ be the triangulation of $C_1$ constructed as in the step 1. of the construction given above for $\epsilon=\epsilon_i$.  Let $\mathcal{U}_{\epsilon_{i}}$ denote the cover of $C_1$ obtained by performing step 2. and step 3. of the above construction on $T_i$. Let $S_i$ be the subcomplex of $T_i$ spanned by vertices of $T_i$ contained in $D$. Thus, $S$ is a triangulation of $D$. For $i=1,2$, let $\mathcal{V}_{\epsilon_{i}}$ denote the covers of $D$ obtained by performing step 2. and step 3. of the construction given above on the triangulation $S_i$. Then $N(\mathcal{V}_{\epsilon_{i}})$ is naturally a subcomplex of $\partial N(\mathcal{U}_{\epsilon_{i}})$.
   \\
  To simplify notation we denote $\partial N(\mathcal{U}_{\epsilon_{i}})$ by $X_i$ for $i=1,2$ and we denote the subcomplexes $N(\mathcal{V}_{\epsilon_{i}})$ by $M_i$. Let $N_i$ denote the boundary $\partial N(\mathcal{V}_{\epsilon_i})$. By assumption, $p_\ast:\widetilde{H}_{d-1}(M_1,N_1)\rightarrow \widetilde{H}_{d-1}(M_2, N_2)$ is an isomorphism. Let $U_i$ be an open neighborhood of $M_i$
 such that $U_i$ deformation retracts onto $M_i$ and $ p^{-1}(U_2)=U_1$. Such a pair of neighborhoods exists by Corollary \ref{cor:nbd}. Let $V_i$ denote the set $U_i\setminus int(M_i)$. Let $L_i$ denote the set $X_i\setminus int(M_i)$. Consider the following commutative diagram where the two horizontal rows correspond to the long exact sequences of the pairs $(X_1,L_1) $ and $(X_2,L_2)$ respectively.
   \begin{center}
\begin{tikzcd}
\\
0 \simeq \widetilde{H}_{d-1}(L_1)\arrow[d, ]\arrow[r,swap,]&\widetilde{H}_{d-1}(X_1)\arrow[d, ]\arrow[r,swap]
  & \widetilde{H}_{d-1}(X_1,L_1)\arrow[d, ] \arrow[r,swap,]
  & \widetilde{H}_{d-2}(L_1)\arrow[d, ] \simeq 0  \\
0 \simeq \widetilde{H}_{d-1}(L_2)\arrow[r,swap,]&\widetilde{H}_{d-1}(X_2)\arrow[r,swap]
  & \widetilde{H}_{d-1}(X_2,L_2) \arrow[r,swap,]
  & \widetilde{H}_{d-2}(L_2) \simeq 0
\end{tikzcd}
\end{center}
Note that $L_i$'s are homeomorphic to $\mathbb{D}^{d-1}$ and hence their (reduced) homology groups vanish in all dimensions.
It follows that $p_\ast:\widetilde{H}_{d-1}(X_1)\rightarrow \widetilde{H}_{d-1}(X_2)$ is an isomorphism if and only if $p_\ast:\widetilde{H}_{d-1}(X_1,L_1)\rightarrow \widetilde{H}_{d-1}(X_2,L_2)$ is an isomorphism.
 The retraction that takes $U_i$ to $M_i$ restricts to a retraction of $V_i$ onto $N_i$. It follows that the canonical inclusion $ j_{i}:(M_i,N_i)\hookrightarrow(U_i,V_i)$ induces an isomorphism $\widetilde{H}_{d-1}(M_i,N_i)\hookrightarrow \widetilde{H}_{d-1}(U_i,V_i)$ on the corresponding homology groups.
 By the Excision principle, for $i=1,2$, the maps $k_i : \widetilde{H}_{d-1}(U_i,V_i) \rightarrow \widetilde{H}_{d-1}(X_i,L_i)$ induced by the canonical inclusion maps $k_i:(U_i,V_i)\hookrightarrow(X_i,L_i)$ are isomorphisms.
    Moreover, the following diagram commutes:
    \begin{center}
    \begin{tikzcd}
   (M_1,N_1)\arrow{r}{k_1\circ j_1}\arrow{d}{p}& (X_1,L_1)\arrow{d}{p}\\
   (M_2, N_2) \arrow{r}{k_2\circ j_2}&(X_2,L_2) \\
    \end{tikzcd}
    \end{center}
    Thus we get the following commutative square on the respective homology groups:
    \begin{center}
   \begin{tikzcd}
 \widetilde{H}_{d-1}(M_1,N_1) \arrow{r}{ k_{1\ast}\circ j_{1\ast}}\arrow{d}{p_\ast}&  \widetilde{H}_{d-1}(X_1,L_1)\arrow{d}{p_\ast}\\
 \widetilde{H}_{d-1}(M_2,N_2)\arrow{r}{k_{2\ast}\circ j_{2\ast}}& \widetilde{H}_{d-1}(X_2,L_2) \\
    \end{tikzcd}
    \end{center}
    It follows that the map $p_\ast: \widetilde{H}_{d-1}(X_1,L_1) \rightarrow \widetilde{H}_{d-1}(X_2,L_2)$ is an isomorphism if and only if $p_\ast: \widetilde{H}_{d-1}(M_1,N_1) \rightarrow \widetilde{H}_{d-1}(M_2,N_2)$ is an isomorphism.
    The statement now follows.
\end{proof}

We use the above results to prove the following result about covers of cubes of arbitrary side length satisfying certain prescribed properties:  

\begin{theorem}\label{thm:cover}
    Let $C_L$ denote the $d$-dimensional cube of side length $L$, $C_L=[0,L]^{d}$, equipped with the $l_{\infty}$-metric. There exists a constant $M=M(d)\in (0,1)$ such that for every $0<\epsilon<1$, there exist a cover $\mathcal{V}_{\epsilon}$ of $C_L$ satisfying the following properties:
    \begin{enumerate}
        \item $\operatorname{mesh}(\mathcal{V}_{\epsilon})<\epsilon L$
        \item $\mathcal{L}(\mathcal{V}_\epsilon)\geq\epsilon LM$
        \item   $|N(\mathcal{V}_{\epsilon})|$ is homeomorphic to $\mathbb{D}^{d}$. In particular, $ \mathcal{V}_{\epsilon}$ has multiplicity $d+1$.
        \item If $\epsilon_1,\epsilon_2>0$ be such that $\epsilon_1 < \frac{\epsilon_2}{M}$, so that $\mathcal{V}_{\epsilon_{1}}\preccurlyeq \mathcal{\mathcal{V}}_{\epsilon_{2}}$. Then, for any refinement map $p$, $p(\partial N(\mathcal{V}_{\epsilon_{1}}))\subseteq \partial N(\mathcal{V}_{\epsilon_{2}})$ and  $$p_{\ast}:H_{d}(N(\mathcal{V}_{\epsilon_{1}}),\partial N(\mathcal{V}_{\epsilon_{1}}))\rightarrow H_d( N(\mathcal{V}_{\epsilon_{2}}),\partial N(\mathcal{V}_{\epsilon_{2}}))$$ is an isomorphism.
        \item Let $[U_{0},U_{1},\dots,U_{t}]$ be a simplex of $N(\mathcal{V}_{\epsilon})$ such that $d(\cap_{i=0}^{k}U_i,\partial C_L)< \frac{\epsilon L}{(d+1)10}$. Then, $[U_{0},U_{1},\dots,U_{t}]$ lies in $\partial N(\mathcal{V}_{\epsilon})$.
    \end{enumerate}

\end{theorem}

\begin{proof}
The above construction shows that for $L=1$ and for any $\epsilon\in (0,1)$ one can construct a cover $\mathcal{U}_\epsilon$ such that the above properties are satisfied. Let $\delta_L:C_1\rightarrow C_L$ denote the dilation $\delta_L(x)=Lx$.
For arbitrary $L$, 
We start with the cover $\mathcal{U}_\epsilon$ and apply  $\delta_{L}$ to obtain the cover $\mathcal{V}_\epsilon$ of $C_L$:  
$$\mathcal{V}_\epsilon=\{\delta_L(U)|U\in \mathcal{U}_\epsilon\}$$
It can be easily verified that  $\textrm{mesh}(\mathcal{V}_\epsilon)=\textrm{mesh}(\mathcal{U}_\epsilon)\cdot L<\epsilon L$ and $\mathcal{L}(\mathcal{V}_\epsilon)=\mathcal{L}(\mathcal{U}_\epsilon)\cdot L \geq\frac{\epsilon L}{M}$. Properties (3),(4) and (5) can be obtained from Lemmas \ref{lemtwo}, \ref{lemthree}, \ref{lem:boundary} after suitably adapting the above given proof and taking the rescaling into account. 
\end{proof}

\section{Proof of Theorem \ref{mainresult}}\label{section:mainresult}
In this section we give a proof of Theorem \ref{mainresult} using results obtained in Section \ref{section:regnbd} and Section \ref{section:covers}. We first outline the idea of the proof.
\\
 \textit{Idea of the Proof}: We show that $\operatorname{asdim}(\mathcal{A})=a\geq d-1$. For the sake of contradiction assume that $a\leq d-2$. Let $g$ denote the function that satisfies the conclusion of Lemma \ref{lemmahomology} and let $\lambda=g(1)$. For a large enough $N\in \mathbb{N}$, we show that there exist covers $\mathcal{U}_1$ and $\mathcal{U}_2$ such that $\operatorname{mesh}(\mathcal{U}_1)<1$ and $\mathcal{L}(\mathcal{U}_2)>\lambda$ and subcomplexes $F_1$ and $F_2$ of $N(\mathcal{U}_1)$ and $ N(\mathcal{U}_2)$ respectively such that for some refinement map $p$, $p(F_1)\subseteq F_2$ and the induced map $p^\ast:H^{k}(N(\mathcal{U}_2),F_2)\rightarrow H^{k}(N(\mathcal{U}_1),F_1)$ is nontrivial for some $k>d-2$. To produce the covers $\mathcal{U}_1,\mathcal{U}_2$, we first produce appropriate covers $\mathcal{X}_1,\mathcal{X}_2$ of $C_N$ and take their restriction to $A_N$. For this we use Theorem \ref{thm:cover} derived in section \ref{section:covers}. There are natural subcomplexes $B_1$ and $B_2$ of $N(\mathcal{X}_1)$ and $N(\mathcal{X}_2)$ respectively such that $B_i$ is homotopy equivalent to $N(\mathcal{U}_i)$ for $i=1,2$ (Lemma \ref{lemma:relation}). Thus the problem can be reduced to finding appropriate subcomplexes $K_i$ of $B_i$ such that any refinement map $f:\mathcal{X}_1\rightarrow \mathcal{X}_2$ induces a nontrivial map $p^{\ast}:H^{k}(B_2,K_2)\rightarrow H^{k}(N(B_1,K_1)$ for some $k\in\{d-1,d\}$. The main tool we use to produce cohomology classes in $H^{k}(B_i,K_i)$ is the Alexander Duality theorem (Theorem \ref{poincarelefschetz}). Our covers $\mathcal{X}_i$ are constructed such that their nerves $N(\mathcal{X}_i)$ are homeomorphic to $\mathbb{D}^d$. Theorem \ref{poincarelefschetz}  cannot directly be applied to $B_i$ as these subcomplexes may not contain the boundary. Therefore, we first consider the  subcomplexes $D_i=B_i\cup \partial N (\mathcal{X}_i)$. We show that each $D_i$ topologically separates $N(\mathcal{X}_i)$ i.e. $\operatorname{rank}(H_{0}(N(\mathcal{X}_i)\setminus D_i))\geq 2$ (Lemma \ref{lemma:component}). Applying the Theorem \ref{poincarelefschetz}, this gives us, for each $i=1,2$, at least two (linearly independent) cohomology classes in $H^{d-1}(D_i)$ (Corollary \ref{rankcorollary}). Furthermore the cohomology classes in $H^{d-1}(D_2)$ are the images of the classes in $H^{d-1}(D_2)$ under the map $f^{\ast}:H^{d}(D_2)\rightarrow H^{d-1}(D_1)$. As $D_i$ is the union $B_i\cap \partial N(\mathcal{U}_i)$, the cohomology class $H^{d-1}(\partial N(\mathcal{U}_i))\simeq H^{d-1}(\mathbb{S}^{d-1})$ contributes one towards the rank of $H^{d-1}(D_i)$. We show that the other contribution must come from nontrivial classes in $H^{k}(B_i,K_i)$, for some $k\in\{d,d-1\}$, where $K_i=B_i\cap \partial N(\mathcal{U}_i)$. To show the nontriviality of the map $p^{\ast}:H^{k}(B_2,K_2)\rightarrow H^{k}(B_1,K_1)$ one needs that these cohomology classes are preserved under $p^{\ast}$. 
\vspace{5mm}
\par
We now give the proof of Theorem \ref{mainresult} in more detail:
 \begin{proof}[Proof of Theorem \ref{mainresult}]
Assume that $\operatorname{asdim}(\mathcal{A})=a\leq d$ and let $g:\mathbb{R}_+\rightarrow \mathbb{R}_+$ be a function such that the conclusion of Lemma \ref{lemmahomology} is satisfied. 
Let $\lambda=g(1)$. We show that $a\geq d-1$. Let $M$ be the constant given by Theorem \ref{thm:cover}.
    Let  $N\in \mathbb{N}$ be such that $L_{N}>b_{N}>\frac{\lambda}{M}$.  By applying Theorem \ref{thm:cover} to $\epsilon = \frac{\lambda}{L_{N}M}<1$, we obtain a  cover $\mathcal{X}_2$ of $C_{N}$ such that $\operatorname{mesh}(\mathcal{X}_2)<M\lambda$, and $\mathcal{L}(\mathcal{X}_2)>\lambda$. Using, Theorem \ref{thm:cover} again, we obtain one more  cover $\mathcal{X}_1$ of $C_{N}$ such that $\operatorname{mesh}(\mathcal{X}_1)<\min\{b_{N},1\}$. Covers $\mathcal{X}_1$ and $\mathcal{X}_2$ are assumed to satisfy properties (3) and (4) mentioned in Theorem \ref{thm:cover}. 
   For $i=1,2$, let $\mathcal{U}_i=\{U\cap A_{N}|U\in \mathcal{X}_i\}$ denote the cover of $A_{N}$ obtained by restricting subsets of $\mathcal{X}_1$ to $A_{N}$. Note that $\operatorname{mesh} (\mathcal{U}_1)<\min\{b_{N},1,\frac{\lambda}{10(d+1)M}\}$ and $\mathcal{L}(\mathcal{U}_2)>\lambda$. Henceforth, for $i=1,2$, we denote $N(\mathcal{X}_i)$ by $X_i$ to simplify notation. We will fix a refinement map $N(\mathcal{X}_1)\rightarrow N(\mathcal{X}_2)$ and denote it by $f$.
    We also define two subcomplexes $B_1$ and $B_2$ of $X_1$ and $X_2$ respectively. Let $B_i$ be the subcomplex of $X_i$ consisting of all simplices $[U_0,U_1,\dots,U_t]\in X_i$ such that $A_N\cap U_0 \cap U_1\cap \dots \cap U_t\neq \emptyset$.

For $i=1,2$, there are natural simplicial maps  $s_i: N(\mathcal{U}_i) \rightarrow B_i$ and $r_i: B_i \rightarrow  N(\mathcal{U}_i)$ such that $s_i$ is an embedding  and  $r_i$ is a retraction of $B_i$ onto $s_i(N(\mathcal{U}_i))$. For each $P\in \mathcal{U}_i$ we choose an element $U\in \mathcal{X}_i$ such that $P\subseteq U$ and let $s_i$ be the map that sends $P$ to $U$. $s_i$ defines a simplicial map from $N(\mathcal{U}_i)\rightarrow X_i$ whose image is contained in $B_i$. Define $r_i$ to be the map that sends an element $U\in B_i^{(0)}$ to $U\cap A_N \in N(\mathcal{U}_i)$. 

\begin{lemma}\label{lemma:relation}
Let $s_i, r_i$ be as above. Then:
\begin{enumerate}
    \item $s_i \circ r_i = id_{B_i}$ and $r_i \circ s_i$ is contiguous to the identity map $id_{N(\mathcal{U}_i)}$. 
    \item For any refinement map $p_{\mathcal{U}_1}^{\mathcal{U}_2}$, the following diagram commutes up to contiguity:
        \begin{center}
        \begin{tikzcd}
            N(\mathcal{U}_1) \arrow[dd,"s_1"] \arrow[rr,"p_{\mathcal{U}_1}^{\mathcal{U}_2}"] & & N(\mathcal{U}_2) \\
            & & \\
            B_1 \arrow[rr,"f = p_{\mathcal{X}_1}^{\mathcal{X}_2}"] & & B_2 \arrow[uu,"r_2"]
        \end{tikzcd}
        \end{center}
\end{enumerate}
\end{lemma}

\begin{proof}
 Let $P \in N(\mathcal{U}_i)^{(0)}$. Then $P = U \cap A_N$ for some (possibly non-unique) $U_i \in \mathcal{X}_i$. The map $s_i$ sends $P$ to one such element $U \in B_i^{(0)}$ with $P \subseteq U$. Moreover, $r_i(U) = U \cap A_N = P.$ Therefore, $r_i \circ s_i = id_{N(\mathcal{U}_i)}$.
\\
Let $\sigma = [U_0, U_1, \dots, U_n]$ be a simplex in $N(\mathcal{U}_i)$. Then
$A_N \cap U_0 \cap \dots \cap U_n \neq \emptyset$.
Let
$r_i \circ s_i(\sigma) = [V_0, V_1, \dots, V_n].$
For each $0 \le i \le n$, we have $U_i \cap A_N = V_i \cap A_N$. Hence, $A_N \cap (\bigcap_{j=0}^{n} U_j) \cap (\bigcap_{j=0}^{n}V_j) \neq \emptyset.$ This shows that the set $\{U_0,\dots,U_n,V_0,\dots,V_n\}$ spans a (possibly degenerate) simplex in $B_i$.

By Lemma~\ref{simplicialhomotopy}, the identity map $id_{N(\mathcal{B}_i)}$ is contiguous to $r_i \circ s_i$.

Now for part~(2). By Lemma~\ref{relativehomotopy}, it suffices to show that $r_2 \circ f \circ s_1$ is a refinement map.

Let $P \in N(\mathcal{U})^{(0)}$ and let $j_1(P) = U$. Then $P = A_N \cap U$. Let $f(U) = V$ and $r_2(V) = Q$. By definition of $f$ and $r_2$, we have $U \subseteq V$ and $A_N \cap V = Q$. Thus,
$$
P = A_N \cap U \subseteq Q = A_N \cap V,
$$
and hence $r_2 \circ f \circ s_1$ is a refinement map.
\end{proof}

We define a few more subcomplexes of $X_1$ and $X_2$. 
Let $\partial X_i$ denote the topological boundary of $X_i$. Both $\partial X_1$ and $\partial X_2$ are homeomorphic to $\mathbb{S}^{d-1}$ and $f(\partial X_1)\subseteq \partial X_2$ (Theorem \ref{thm:cover}). 
   Let $D_{i}$ denote the subcomplex $B_{i} \cup \partial X_i$. 
 The main reason behind defining these subcomplexes is that Theorem \ref{poincarelefschetz} cannot be applied to $B_1$ and $B_2$ directly as these subcomplexes may not necessarily contain the respective boundaries.  \\
We fix vertices $O_1 \textrm{ and } P_1$ of $X_1$ such that $x_N\in O_1$ and $y_N\in P_1$. Let $O_2=f(O_1)$ and $P_2=f(P_1)$ so that $x_N\in O_2 \supset O_1$ and  $y_N\in P_2 \supset P_1$. We show that $X_1 \setminus D_1 $ (resp. $X_2\setminus D_2$) contains more than one connected component and that $O_1\textrm{ and }P_1$ (resp. $O_2$ and $P_2$) lie in different connected components: 
  
\begin{lemma}\label{lemma:component}
For $i=1,2$, $O_i$ and $P_i$ lie in distinct connected components of $X_i \setminus D_i$. 
\end{lemma}

\begin{proof}
We first show that $O_i$ and $P_i$ lie in the complement of $D_i$. Since $\operatorname{mesh}(\mathcal{X}_i)<b_N$ and $d(x_N,A_N\cup \partial C_N)>b_N$, the sets $O_i,P_i$ intersect neither $ \partial C_N$ nor $A_N$. It follows that  the vertices $O_i$ and $P_i$ lie in the complement of $D_i$. 
       We now show that $O_i$ and $P_i$ lie in distinct connected components. By Lemma \ref{connected}, it suffices to prove that every edge-path joining $O_i$ to $P_i$ in $Bd(X_i)$ intersects $D_i$. Let $ O_i=w_1,w_2,\dots ,w_l=P_i$ be a sequence of vertices in $Bd(X_i)^{(0)}$ such that $w_j$ and $w_{j+1}$ are adjacent in $Bd(X_i)^{(0)}$ for each $1\leq j \leq l$. We first construct a path $\gamma$ in $C_N$ that joins $x_N$ to $y_N$. Every vertex $v\in Bd(X_i)^{(0)}$ is the barycenter of a unique simplex $\sigma_{v}=[U_0,U_1,\dots,U_{t-1}]\in X_i$. For each $ 1\leq j\leq l$, let $\sigma_j$ be the simplex such that $w_j$ is the barycenter of $\sigma_j$. For each $j$, either $\sigma_{j}$ is a face of $\sigma_{j+1}$ or vice versa. For each of the simplices $\sigma_j$, $1\leq j\leq l$, let $z_j=z_{\sigma_{j}}$ denote the point associated to $\sigma_j$ as defined in Remark \ref{remark:point}. For each pair $z_j,z_{j+1}$ we choose a path $\gamma_{j}$ that joins $z_j$ to $z_{j+1}$ such that $\gamma_{j}$ is contained in the set $ (\bigcap_{U\in \sigma_{j}}U)\cup (\bigcap_{V\in \sigma_{j+1}}V)$. This is always possible since either $\sigma_j$ is a codimension one face of $\sigma_{j+1}$ or vice versa. We obtain a path $\gamma$ from $x_N$ to $y_N$ by taking the concatenation of these paths $\gamma= \gamma_{1}\star \gamma_2 \star \dots \star \gamma_{l-1}$. Recall that $\operatorname{diam}(O_i),\operatorname{diam}(P_i)<b_N$. As both $x_N$ and $z_1$ lie in $O_i$ it follows that both $x_N$ and $z_1$ lie in the same path-component of $C_N\setminus A_N$. Similarly $y_N$ and $z_l$ belong to the same path-component of $C_N\setminus A_N$. As $z_1$ and $z_l$ lie in different components of $C_N\setminus A_N$, $\gamma$ intersects $A_N$ at least once. Now suppose that $\gamma$ intersects $A_N$ in the subpath $\gamma_k$.  One can assume without loss of generality, after exchanging the role of $x_N$ and $y_N$ if necessary, that $\sigma_k=[V_0,V_1,\dots,V_{t-1}]$ is a face of $\sigma_{k+1}=[V_0,V_1,\dots ,V_{t}]$. By construction, $\gamma_k\subseteq \cap_{j=0}^{t-1}V_j$. This implies that $A_N \cap (\cap_{j=0}^{t-1}V_j)\neq \emptyset$. Consequently, $\sigma_k \in B_i \subseteq D_i$. As a result, $w_k \in D_i$. 
     \end{proof}
   
The following lemma establishes that the image of the map $f^{\ast}:H^{d}(X_2,D_2)\rightarrow H^{d}(X_1,D_1)$ induced by $f$ contains a free abelian subgroup of rank two.

\begin{lemma}\label{mainlemma}
Consider the map $f^{\ast}:H^{d}(X_2,D_2)\rightarrow H^{d}(X_1,D_1)$. 
There exist $\chi,\chi'\in f^\ast(H^{d}(X_2,D_2))$ such that $\operatorname{rank}(\langle \chi,\chi'\rangle)=2$.
\end{lemma}

\begin{proof}
Let $E$ denote the complex $f^{-1}(D_2)$. Then, the map $f:(X_1,D_1)\rightarrow(X_2,D_2)$ factors as $(X_1,D_1)\xrightarrow{m}(X_1,E)\xrightarrow{f} (X_2,D_2)$ where $m$ denotes the inclusion map.

 The maps $\phi_{E}$ and $\phi_{D_2}$ denote the isomorphisms given by the Theorem \ref{poincarelefschetz}. 
  Choose a neighborhood $V$ of $D_2$ such that the following four properties are satisfied: $V$ deformation retracts onto $D_2$, $X_2 \setminus D_2$ deformation retracts onto $X_2\setminus V$, $U:=f^{-1}(V)$ deformation retracts onto $E$, and $ X_1 \setminus U$ deformation retracts onto $X_1 \setminus E$. Such a neighborhood can always be constructed as shown in Corollary \ref{cor:nbd}. We define $X_1^{U}:=X_{1} \setminus U$  and $X_{2}^{V}=X_2\setminus V$. The pairs $(X_1^{U},\partial U)$ and $(X_1^{V},\partial V)$ are $d$-manifolds with boundary. Let $e_{U}\in H_{d}(X_1^{U},\partial U)$ and $e_{V}\in H_{d}(X_1^{V},\partial V)$ denote their fundamental classes. We first prove the following claim:
  
\vspace{5mm}
\textit{Claim: }Let $f_{\ast}:H_{d}(X_{1}^{U},\partial U) \rightarrow H_{d}(X_{2}^{V},\partial V)$ be the map induced on the homology by $f$. Then, $f_\ast(e_U)=e_V$. 
 
 \vspace{5mm}
 \textit{Proof of Claim:}
Consider the following commutative diagram: 
  \begin{center}
\begin{tikzcd}
H_{d}(X_1,\partial X_1)\arrow[r,"u_{1\ast}"]\arrow[d,"f_{\ast}"]
&H_{d}(X_1,\overline{U})\arrow[r,"(k_{1\ast})^{-1}"]\arrow[d,"f_{\ast}"]
&H_{d}(X_{1}^U,\partial U)\arrow[d,"f_{\ast}"]
 \\
H_{d}(X_2,\partial X_2)\arrow[r,"u_{2\ast}"]&H_{d}(X_2,\overline{V})\arrow[r,"(k_{2\ast})^{-1}"] 
  & H_{d}(X_2^{V},\partial{V}) 
\end{tikzcd}
\end{center}
 Here $u_{1\ast}$ is the map induced by the inclusion $u_1:\partial X_1 \hookrightarrow \overline{U}$ and $k_{1\ast}:H_d(X_1^U,\partial {U})\rightarrow H_{d}(X_{1}^U,\overline{U})$ is the map induced by the inclusion $k_1:(X_1^U,\partial {U})\hookrightarrow (X_{1},\overline{U})$. Similarly, $u_{2\ast}$ is the map induced by the inclusion $u_2:\partial X_2 \hookrightarrow \overline{V}$ and $k_{2\ast}:H_d(X_1,\partial {V})\rightarrow H_{d}(X_{1}^V,\overline{V})$ is the map induced by the inclusion $k_2:(X_2^V,\partial {U})\hookrightarrow (X_{2},\overline{V})$. Owing to Excision principle, $k_{1\ast}$ and $k_{2\ast}$ are isomorphisms. 
For $i=1,2$, let $e_i\in H_{d}(X_i,\partial X_i)$ denote the fundamental class of $(X_i,\partial X_i)$. Then, by Theorem \ref{thm:cover}, $f_\ast(e_1)=e_2$. The fundamental classes $e_1$ and $e_U$ satisfy, $(k_{1\ast})^{-1}\circ(u_{1\ast})(e_1)=e_U$. Similarly, $ (k_{2\ast})^{-1}\circ(u_{2\ast})(e_2)=e_V$. 
It follows that $ f_{\ast}(e_U)=e_V$. 
\qed
\\
\hfill \break
 Now consider the following commutative diagram: 

\begin{equation}\label{diagram:daig}
\begin{tikzcd}
H^{k}(X_1,E)\arrow[r,"(t_1^{\ast})^{-1}"]
&H^{k}(X_1,\overline{U})\arrow[r,"p_1^{\ast}"]
&H^{k}(X_{1}^U,\partial U)\arrow[r,"-\cap e_{U}"] 
&H_{d-k}(X_{1}^{U})\arrow[d,"f_\ast"] \arrow[r,"l_{1\ast}"] 
&H_{d-k}(X_1 \setminus E)\arrow[d,"f_{\ast}"] \\
H^{k}(X_2,D_2)\arrow[u,"f^{\ast}"]\arrow[r,"(t_2^{\ast})^{-1}"]&H^{d}(X_2,\overline{V})\arrow[u,"f^{\ast}"]\arrow[r,"p_{2}^{\ast}"] 
  & H^{k}(X_2^{V},\partial{V}) \arrow[u,"f^{\ast}"]\arrow[r,"-\cap e_{V}"] 
  & H_{d-k}(X_2^V) \arrow[r,"l_{2\ast}"] & H_{d-k}(X_2\setminus D_2) 
\end{tikzcd}
\end{equation}
The maps $t_1, t_2, p_1,p_2,l_1,l_2$ are inclusions. 
All of the horizontal maps are isomorphisms. Composing maps in the top row gives the Alexander duality isomorphism $ \phi_{E}$, and similarly, the maps in the bottom row, when composed, give the isomorphism $\phi_{D_2}$.  (cf. Proof Sketch of Theorem \ref{poincarelefschetz}).
\par
The first square commutes as $t_1 \circ f = f\circ t_2$ .
The second square commutes as $p_1 \circ f = f\circ p_2$  Similarly, the fourth square commutes as $ l_1 \circ f=f\circ l_2$. The commutativity of the third square follows from Theorem \ref{naturality} along with the above claim.  By the naturality of the cap product, for all $\varphi \in H^{k}(X_1^{V},\partial V)$, $ f_\ast(e_U)\cap \varphi =f_{\ast}(f^{\ast}(\varphi)\cap e_{U}).$ This establishes the commutativity of the above diagram. 
\par
Let $\alpha$ denote the homology class in $H_{0}(X_1\setminus E)$ that corresponds to the connected component containing the vertex $O_1$. Similarly, let $\beta$ denote the homology class that corresponds to the connected component that contains $ P_1$. $f_{\ast}(\alpha)$ (resp. $f_{\ast}(\beta)$) represents the homology class in $H_{0}(X_2\setminus D_2)$ that corresponds to the connected component containing the vertex $O_2$ (resp. $P_2$). Both $\langle\alpha,\beta\rangle$ and $\langle f_\ast(\alpha),f_\ast(\beta)\rangle$ are free abelian groups of rank two and hence $f_{\ast}: \langle\alpha,\beta\rangle\rightarrow \langle f_{\ast}(\alpha),f_{\ast}(\beta)\rangle$ is an isomorphism. 
\par
Now consider the following diagram: 
\begin{center}
\begin{tikzcd}
H^{d}(X_1,D_1)& H_{0}(X_1 \setminus D_1)\arrow{l}[swap]{(\phi_{D_1})^{-1}}\\
H^{d}(X_1,E)\arrow{u}{m^\ast}&\langle \alpha,\beta \rangle \arrow{u}{m_\ast}\arrow{l}[swap]{(\phi_{E})^{-1}}\\
H^{d}(X_2,D_2)\arrow{u}{f^{\ast}} & \langle f_{\ast}(\alpha),f_{\ast}(\beta) \rangle\arrow{u}{(f_\ast)^{-1}}\arrow{l}[swap]{(\phi_{D_2})^{-1}}\\
\end{tikzcd}
\end{center}
The upper square is derived using the naturality property mentioned in Theorem \ref{poincarelefschetz}. The lower square can be derived using the diagram \ref{diagram:daig}. Recall that $m:X_1\setminus E \rightarrow X_1 \setminus D_1$ is the inclusion map. Thus $m_{\ast}$ maps $\alpha$ (resp. $\beta$) to the homology class that corresponds to the connected component of $X_1\setminus D_1$ that  contains $O_1$ (resp. $P_1$). As these components are distinct(Lemma \ref{lemma:component}) we have that $ \operatorname{rank}(\langle m_\ast(\alpha), m_\ast(\beta)\rangle)=2$. Let $\chi=(\phi_{D_{1}})^{-1}(\alpha)$ and $\chi'=(\phi_{D_{1}})^{-1}(\beta)$. By commutativity of the above diagram, $\chi ,\chi'\in f^{\ast}(H^{d}(X_2,D_2))$. As $(\phi_{D_{1}})^{-1}$ is an isomorphism, $\operatorname{rank}(\langle \chi,\chi'\rangle)=2$. 
\end{proof} 

\begin{corollary}\label{rankcorollary}
Consider the map $f^\ast:H^{d-1}(D_2)\rightarrow H^{d-1}(D_1)$  map induced by $f$. 
There exist $\theta,\theta'\in f^\ast(H^{d-1}(D_2))$ such that $\operatorname{rank}(\langle \theta,\theta'\rangle)=2$.
\end{corollary}

\begin{proof}
Consider the following commutative diagram:
\begin{center}
\begin{tikzcd}
\\
0\simeq H^{d}(X_1)&H^{d}(X_1,D_1)\arrow[l,swap]
  & H^{d-1}(D_1) \arrow[l,swap, "\delta_1"]
  & H^{d-1}(X_1)\simeq 0 \arrow[l,swap,] \\
0 \simeq H^{d}(X_2)\arrow[u,"f^\ast"]&H^{d}(X_2,D_2)\arrow[l,swap]\arrow[u,"f^\ast"]
  & H^{d-1}(D_2) \arrow[l,swap,"\delta_2"] \arrow[u,"f^{\ast}"]
  & H^{d-1}(X_2)\simeq 0 \arrow[l,swap] \arrow[u,"f^\ast"]
\end{tikzcd}
\end{center}
The two horizontal rows are long exact sequences corresponding to the pairs $(X_1,D_1)$ and $(X_2,D_2)$ respectively.
Recall that both $X_1$ and $X_2$ are homeomorphic to the disk $\mathbb{D}^d$ and hence they have trivial cohomology in all dimensions greater than zero. It follows that $\delta_1$ and $\delta_2$ are isomorphisms. Let $\theta = (\delta_1)^{-1}(\chi), \theta' = (\delta_1)^{-1}( \chi')$. As $\delta_1$ is an isomorphism $\textrm{rank}(\langle \theta,\theta'\rangle)=2$. Let $ \eta,\eta'\in H^{d-1}(X_2,D_2)$ such that $f^{\ast}(\eta)=\theta$ and $f^{\ast}(\eta')=\theta'$. By the commutativity of the above diagram, $\theta = f^\ast((\delta_2)^{-1}(\eta)), \theta' = f^\ast((\delta_2)^{-1}(\eta'))$. As a result, $ \theta,\theta'\in  f^\ast(H^{d-1}(D_2))$. 
\end{proof}

\begin{lemma}\label{lemma:D2X2}
    The map $f^\ast:H^{k}(D_2,\partial X_2)\rightarrow H^{k}(D_1,\partial X_1)$ is nontrivial for at least one out of $k=d$ or $k=d-1$. 
\end{lemma}

\begin{proof}
For the sake of contradiction, assume that the maps induced by $f$ in degrees $d$ and $d-1$ are both trivial. Consider the following commutative diagram:
\begin{center}
\begin{tikzcd}
\\
H^{d}(D_1,\partial X_1)&H^{d-1}(\partial X_1)\arrow[l,swap]
  & H^{d-1}(D_1) \arrow[l,swap]
  & H^{d-1}(D_1,\partial X_1) \arrow[l,swap] \\
H^{d}(D_2,\partial X_2)\arrow[u,"f^\ast"]&H^{d-1}(\partial X_2)\arrow[l,swap]\arrow[u,"f^\ast"]
  & H^{d-1}(D_2) \arrow[l,swap] \arrow[u,"f^{\ast}"]
  & H^{d-1}(D_2,\partial X_2) \arrow[l,swap] \arrow[u,"f^\ast"]
\end{tikzcd}
\end{center}
The two horizontal rows are the long exact sequences associated to the pairs $(D_i,\partial X_i)$.
By assumption, the first and fourth vertical arrows are trivial. As a consequence, we obtain the following exact sequence:
\begin{center}
\begin{tikzcd}
0 & f^{\ast}(H^{d-1}(\partial X_2)) \arrow[l,swap]
  & f^{\ast}(H^{d-1}(D_2)) \arrow[l,swap]
  & 0 \arrow[l,swap]
\end{tikzcd}
\end{center}
Note that, by Theorem \ref{thm:cover} (3), $f^{\ast}(H^{d-1}(\partial X_2)) \cong \mathbb{F}_2$ .
This yields as contradiction, while by Corollary~\ref{rankcorollary} we have
$$
\operatorname{rank}(f^{\ast}(H^{d-1}(D_2))) \ge 2.
$$
\end{proof}
For $i=1,2$, let $K_i$ denote the subcomplex $B_i\cap \partial X_i$. 
The following lemma shows that the map induced by $f$ between the relative cohomology of the pairs $(B_2,K_2)$ and $(B_1,K_1)$ is nontrivial in either the top degree $d$ or one dimension below.

\begin{lemma}\label{lemma:penultimate}
The map $f^\ast \colon H^{k}(B_2,K_2)\rightarrow H^{k}(B_1,K_1)$ is nontrivial for at least one of $k=d$ or $k=d-1$.
\end{lemma}

\begin{proof}
   For $i=1,2$, let $W_i$ denote an open neighborhood of $K_i$ relative to $D_i$ such that $W_i$ deformation retracts onto $K_i$. Such a relative neighborhood exist by Corollary \ref{cor:nbd}. Let $R_i$ denote the open set $B_i\cup W_i$. The retraction that collapses $W_i$ onto $B_i$ extends to a retraction of $ R_i$ onto $B_i$. Thus, the inclusion map $(B_i,K_i)\hookrightarrow (R_i,W_i)$ induces an isomorphism $H^{k}(B_i,K_i)\rightarrow H^{k}(R_i,W_i)$. Furthermore, by Excision, the inclusion map $(R_i,W_i)\hookrightarrow (D_i,\partial X_i)$ induces an isomorphism $H^{k}(R_i,W_i)\rightarrow H^{k}(D_i,\partial X_i)$.
    As a result, for $i=1,2$, the map  $h_{i}^{\ast}:H^{k}(D_i,\partial X_i)\rightarrow H^{k}(B_i,K_i)$ induced by the canonical inclusion map $h_i:(B_i,K_i)\hookrightarrow(D_i,\partial X_i)$ is an isomorphism. 
    Moreover, the following diagram commutes, 
    \begin{center}
    \begin{tikzcd}
    (B_1,K_1)\arrow{r}{h_1}\arrow{d}{f}& (D_1,\partial X_1)\arrow{d}{f}\\
    (B_2,K_2)\arrow{r}{h_2}& (D_2,\partial X_2)\\
    \end{tikzcd}
    \end{center}
    Thus we get the following commutative square on the respective cohomology groups:  
    \begin{center}
    \begin{tikzcd}
    H^k(D_2,\partial X_2)  \arrow{r}{h_2^{\ast}}\arrow{d}{f^\ast}&              H^k(B_2,K_2)\arrow{d}{f^\ast}\\
    H^k(D_1,\partial X_1)\arrow{r}{h_1^\ast}& H^k(B_1,K_1)\\
    \end{tikzcd}
    \end{center}
    The statement now follows from Lemma \ref{lemma:D2X2}. 
\end{proof}
In the next couple of lemmas we infer cohomological non-vanishing of refinement maps between covers $\mathcal{U}_i$ using the results proved above. 
Let $s_1,s_2,r_1,r_2$ be as in Lemma \ref{lemma:relation}. For $i=1,2$, let $F_i$  denote the complex $s_{i}^{-1}(K_i)$.
In the next lemma we verify that $F_1$ maps to  $F_2$ under any refinement map between $\mathcal{U}_1$ and $\mathcal{U}_2$. This condition is essential if one wishes to use Lemma \ref{lemmahomology}.
\begin{lemma}
$p_{\mathcal{U}_1}^{\mathcal{U}_2}(F_1) \subseteq F_2$ for any refinement map $p_{\mathcal{U}_1}^{\mathcal{U}_2}$. Furthermore, for any two refinement maps $p_{\mathcal{U}_1}^{\mathcal{U}_2},q_{\mathcal{U}_1}^{\mathcal{U}_2}$ and for any simplex $ \sigma=[U_0,U_1,\dots,U_t]\in F_1$ the simplex spanned by $\{p_{\mathcal{U}_1}^{\mathcal{U}_2}(U_i)\}_{i=0}^{t}\cup\{q_{\mathcal{U}_1}^{\mathcal{U}_2}(U_t)\}_{i=0}^{t}$ lies in $F_2$.\end{lemma}

\begin{proof}
 Let $[U_0,U_{1},\dots ,U_{t}]$ be a simplex in $F_1$. Then, by assumption, the simplex \\$[s_1(U_0),s_1(U_{1}),\dots ,s_1(U_{t})]$  lies in $K_1=\partial X_1 \cap B_1$. If follows that the set $\cap_{i=0}^{t}s_1(U_i)$ intersects $ \partial C_N$ nontrivially. Since $ \cap_{i=0}^{t}U_i\subset \cap_{i=0}^{t}s_1(U_i)$, and  $\operatorname{diam}(\cap_{i=0}^{t}s_1(U_i))\leq \frac{\lambda}{M(d+1)20}$, we have that  $d(\cap_{i=0}^{t}U_i,\partial C_N)< \frac{\lambda}{M(d+1)10}$. Let  $p=p_{\mathcal{U}_1}^{\mathcal{U}_2}$ be a refinement map and    let $V_i=p(U_i)$, $0\leq i\leq t$. To show that $[V_0,V_1,\dots,V_t]$ lies in $F_2$ it suffices to prove that $[s_2(V_0),s_2(V_1),\dots,s_2(V_t)]\in \partial X_2$. By Theorem \ref{thm:cover}(5) it is enough to show that $d(\cap_{i=0}^{t}s_2(V_i),\partial C_{N})< \frac{\epsilon L_N}{10(d+1)}=\frac{\lambda}{M(d+1)10}$ (Recall that  $\epsilon= \frac{\lambda}{L_N M}$). We have the following containments, $\cap_{i=0}^{t}U_i\subset  \cap_{i=0}^{t}(V_i)\subset \cap_{i=0}^{t}s_2(V_i)$. As a result, $d(\cap_{i=0}^{t}s_2(V_i),\partial C_N)\leq d(\cap_{i=0}^{t}s_2(V_i),\partial C_N) < \frac{\lambda}{M(d+1)10}$. 
 \par 
 Now let $q=q_{\mathcal{U}_1}^{\mathcal{U}_2}$ be another refinement map. Let $W_i=q(U_i)$, $0\leq i\leq t$. Let $\sigma'$ be the simplex spanned by $\{V_i\}_{i=0}^{t}\cup\{W_i\}_{i=0}^{t}$. To show that $\sigma'$ lies in $F_2$ it suffices to show that $s_2(\sigma')$ lies in $K_2$. To prove this we again use Theorem \ref{thm:cover}(5) and show that $d((\cap_{i=0}^{t} s_2(V_i))\cap (\cap_{i=0}^{t}s_2(W_i)),\partial C_{N})< \frac{\epsilon L_N}{10(d+1)}=\frac{\lambda}{M(d+1)10}$. As observed earlier, $d(\cap_{i=0}^{t}U_i,\partial C_N)\leq \frac{\lambda}{M(d+1)10}$. The statement now follows from the following series of  containments, $\cap_{i=0}^{t}U_i\subset   (\cap_{i=0}^{t}V_i) \cap (\cap_{i=0}^{t}W_i)\subset (\cap_{i=0}^{t}s_2(V_i))\cap (\cap_{i=0}^{t}s_2(W_i))$.
\end{proof}

Finally, we prove that any refinement map $p:N(\mathcal{U}_1)\rightarrow N(\mathcal{U}_2)$ the induced map $p^{\ast}: H^{k}(N(\mathcal{U}_2),F_2)\rightarrow H^{k}(N(\mathcal{U}_1),F_1)$ is non-trivial for some $k>d-2$. 
\begin{lemma}
  Any refinement map $p=p_{\mathcal{U}_1}^{\mathcal{U}_2}$ induces a nontrivial map $$p^{\ast}: H^{k}(N(\mathcal{U}_2),F_2)\rightarrow H^{k}(N(\mathcal{U}_1),F_1)$$  for at least one of $k=d$ or $k=d-1$. 
\end{lemma}

\begin{proof}
 By Lemma \ref{lemma:relation}(1) the following diagram commutes:
  \begin{center}
    \begin{tikzcd}
    (N(\mathcal{U}_1),F_1)\arrow{d}{s_{1}}\arrow{r}{p}&              (N(\mathcal{U}_2),F_2)\\
    (B_1,K_1)\arrow{r}{f}& (B_2,K_2)\arrow{u}{r_{2}}\\
    \end{tikzcd}
    \end{center}
    This induces the following diagram on the respective relative cohomology groups:

\begin{center}
    \begin{tikzcd}
   H^{k}(N(\mathcal{U}_1),F_1)&
    H^{k}(N(\mathcal{U}_2),F_2)\arrow{d}[swap]{r_{2}^*}\arrow{l}[swap]{p^*}\\
     H^{k}(B_1,K_1)\arrow{u}[swap]{s_{1}^*}& H^{k}(B_2,K_2)\arrow{l}[swap]{f^*}
    \\
    \end{tikzcd}
\end{center}
We claim that both vertical maps are isomorphisms. Consider the maps $s_{i}:(N(\mathcal{U}_i),F_i)\rightarrow (B_i,K_i)$ and $r_{i}:(B_i,K_i)\rightarrow (N(\mathcal{U}_i),F_i)$. We have $r_{i}\circ s_i=id_{N(\mathcal{U}_{i})}$ and $s_{i}\circ r_i$ is contiguous  to $id_{B_i}$. Consequently, $s_{i}^{\ast}\circ r_{i}^{\ast}=\operatorname{id}$ and  $r_{i}^{\ast}\circ s_{i}^{\ast}=\operatorname{id}$. Thus $r_i^\ast$ and $s_i^{\ast}$ are isomorphisms. The statement now follows from Lemma \ref{lemma:penultimate}. 
\end{proof}

    This, by Lemma \ref{lemmahomology}, shows that $a\geq d-1$.
    
\end{proof}

\section{An application to coarse embeddings}\label{applications}

In this section we give an application of our result to ruling out coarse embeddings $f:X\rightarrow Y$  where $X$ satisfies $QF_d$ and $Y$ is coarsely separated by a family of subspaces that have asymptotic dimension at most $d-2$. 

 \begin{corollary}\label{embeddingcorollary}
    Let $Z$ be a metric space with  property  $\textrm{QF}_d$ and let $X$ be a graph of spaces associated to a finite  graph of groups $(\Gamma,\mathcal{A})$ such that all the edge groups of $\Gamma$ have asymptotic dimension $\leq d-2$. Then the image of any coarse embedding $f:Z \rightarrow \tilde{X}$ is contained in a neighborhood of a vertex space. 
\end{corollary}

\begin{proof}
 Each edge space is coarsely equivalent to one of the edge subgroups $G_e$. It follows that $\operatorname{asdim}(\mathcal{E})=\operatorname{max}_{e\in E\Gamma}\{G_e\}\leq d-2$. By Theorem \ref{thm:coarsesepembedding}, it suffices to rule out the possibility that the image $f(Z)$ of $f$ is coarsely separated by the collection $\mathcal{E}=\{E\cap f(Z)| E \textrm{ is a edge space}  \}$. Suppose that the family $\mathcal{E}$ coarsely separates $f(Z)$. We can assume $f$ to be a continuous map. Let $L\geq 0$ be such that given any $D>0 $ there exist $E\in \mathcal{E}$ and points $ x,y \in f(Z)$ that lie in distinct components of  $f(Z)\setminus N_L(E)$ such that $d(x,E),d(y,E)>D$. We first prove the following claim:
 
 \vspace{5mm}

 \textit{Claim:} The family $\{f^{-1}(N_L(E))| E\in \mathcal{E}\}$ coarsely separates $Z$. 
  \vspace{5mm}

 \textit{Proof of the Claim:} 
 Let $D>0$. Let $D'>0$ be such that for any $x,y\in X$ with $d(x,y)>D'$ we have $d(f^{-1}(x),f^{-1}(y))>D$. Let $E\in \mathcal{E}$ such that there exist points $x,y\in f(Z)$ that lie in distinct components of  $f(Z)\setminus N_L(E)$ such that $d(x,E),d(y,E)>D'$. Pick two points $x'\in f^{-1}(x)$ and $y\in f^{-1}(y)$. It follows that $d(x',f^{-1}(E)),d(y',f^{-1}(E))>D'$. Given a continuous path $\gamma$ joining $x'$ to $y'$, $f\circ \gamma$ yields a continuous path that joins $x$ to $y$. Since $f\circ \gamma$ necessarily intersects $N_L(E)$ we have that $\gamma$ must intersect $f^{-1}(N_L(E))$ i.e., $x'$ and $y'$ lie in distinct components of $f^{-1}(N_L(E))$.   
\qedsymbol
 
 \vspace{5mm}

As the family $\{f^{-1}(N_L(E))| E\in \mathcal{E}\}$ coarsely embed in the family $\mathcal{E}$, by Corollary \ref{asdiminvariance}, we have that $\operatorname{asdim}(\{f^{-1}(N_L(E))| E\in \mathcal{E}\})\leq \operatorname{asdim}(\mathcal{E})\leq d-2$. But since every separating family must have asymptotic dimension at least $d-1$ by Theorem \ref{main}, we are able to rule out the first possibility in Theorem \ref{thm:coarsesepembedding}. 
Thus it must be that $f(Z)$ is contained in the neighborhood of some vertex space. 
 
\end{proof}
We now give a concrete example where the above result can be used to rule out the existence of a coarse embedding.
 
 \vspace{5mm} 
\textbf{Example:}
 Let $M$ be an $d$-dimensional hyperbolic manifold, $d\geq 3$,  and let $F$ be a rank-two free subgroup of $\pi_{1}(M)$ generated by $a,b\in \pi_1(M)$ and let $F'$ denote the group $ \langle a,b^2\rangle$. Let $ \phi:F\rightarrow F'$ denote the isomorphism that sends $a$ to $a$ and $b$ to $b^{2}$. Let $G$ denote the HNN extension associated to $\phi$, $$G=\pi_1(M)_{\ast \phi}=\langle \pi_1(M),t| tat^{-1}=a^{2}, tbt^{-1}=b^{2}\rangle.$$Note that $G$ is neither Gromov-hyperbolic nor CAT(0) as it contains a Baumslag-Solitar subgroup.  We claim that there does not exist a coarse embeddings of $\mathbb{R}^d$ into $G$. Let $Y$ be a graph of spaces associated to $G$ with one vertex space isometric to $M$. Let $\tilde{Y}$ denote the universal cover of $Y$ so that each vertex space of $ \tilde{Y} $ is isometric to $\mathbb{H}^d$. 
The asymptotic dimension of the family of edge spaces is one and hence is always less than or equal to $d-2$. 
Thus, by Corollary \ref{embeddingcorollary}, the image of any coarse embedding of $X$ into $\tilde{Y}$ is contained some finite neighborhood of a vertex space $\tilde{Y}_v$. We can rule out this possibility as well. By slightly modifying $f$ one obtains a coarse embedding $g$ of $\mathbb{R}^d$ into $\mathbb{H}^d$: Let $g$ be the map that sends $x$ to the point $y$ in $\tilde{Y}_v$ such that $y$ is closest to $f(x)$. It is well-known that $\mathbb{R}^d$ does not coarsely embed into $\mathbb{H}^d$. This can be shown using separation profiles \cite{Benjamini2012}, for instance.  
\section*{Acknowledgments}
The author would like to thank John Mackay for providing key insights into the problem addressed in this article. The author would  also like to thank David Hume for his comments on an earlier version of this article.   The author would also like to thank Romain Tessera for his interest in the work. 
\\
\bibliographystyle{plain}
\bibliography{bibliography}
\end{document}